\documentclass{amsart}

\usepackage{geometry}
\usepackage[english]{babel}
\usepackage{latexsym, mathrsfs}
\usepackage{amsfonts, amsthm, amsmath, amssymb}
\usepackage{graphicx,verbatim}
\usepackage[table]{xcolor}

\newtheorem{prop}{Proposition}[section]
\newtheorem{lem}[prop]{Lemma}
\newtheorem{cor}[prop]{Corollary}
\newtheorem{thm}[prop]{Theorem}
\theoremstyle{definition}

\newtheorem{defi}[prop]{Definition}
\newtheorem{ex}[prop]{Example}

\def\Z{\mathbb{Z}}
\def\N{\mathbb{N}}

\def\equad{\quad \textrm{ and } \quad}
\def\omb{\cellcolor{black!5}}

\def\H{\mathrm{H}}
\def\supp{\mathsf{supp}}

\def\lcm{\mathrm{lcm}}
\def\B{\mathcal{B}}
\def\F{\mathcal{A}}
\def\G{\mathcal{G}}
\def\Sh{\lfloor h/|S| \rfloor}

\def\wt{\widetilde}

\numberwithin{equation}{section}
\newcommand\con{\mathop{+\mkern-10mu+}}

\begin{document}

\title{On the existence of integer relative Heffter arrays}

\author[F. Morini]{Fiorenza Morini}
\address{Dipartimento di Scienze Matematiche, Fisiche e Informatiche, Universit\`a di Parma,
Parco Area delle Scienze 53/A, 43124 Parma, Italy}
\email{fiorenza.morini@unipr.it}

\author[M.A. Pellegrini]{Marco Antonio Pellegrini}
\address{Dipartimento di Matematica e Fisica, Universit\`a Cattolica del Sacro Cuore, Via Musei 41,
25121 Brescia, Italy}
\email{marcoantonio.pellegrini@unicatt.it}

\begin{abstract}
Let $v=2ms+t$ be a positive integer, where $t$ divides $2ms$, and let
$J$ be the subgroup of order $t$ of the cyclic group $\Z_v$.
An integer Heffter array $\H_t(m,n;s,k)$  over $\Z_v$ relative to $J$ is  an $m\times n$  partially filled
array with elements in $\Z_v$ such that:
(a) each row contains $s$ filled cells and each column contains $k$ filled cells;
(b) for every $x\in \Z_v \setminus J$, either $x$ or $-x$ appears in the array;
(c) the elements in every row and column, viewed as integers in
$\pm\left\{ 1, \ldots, \left\lfloor \frac{v}{2}\right\rfloor \right\}$,  sum to $0$ in $\Z$.

In this paper we study the existence of an integer $\H_t(m,n;s,k)$ when $s$ and $k$ are both even, proving the following 
results.
Suppose that $4\leq s\leq n$ and $4\leq k \leq m$ are such that $ms=nk$.
Let $t$ be a divisor of $2ms$.
(a) If $s,k \equiv 0 \pmod 4$,  there exists an integer $\H_t(m,n;s,k)$.
(b) If $s\equiv 2\pmod 4$ and $k\equiv 0 \pmod 4$, there exists an integer $\H_t(m,n;s,k)$ if and only if 
$m$ is even.
(c) If $s\equiv 0\pmod 4$ and $k\equiv 2 \pmod 4$, then there exists an integer $\H_t(m,n;s,k)$ if and only if 
$n$ is even.
(d) Suppose that $m$ and $n$ are both even. If $s,k\equiv 2 \pmod 4$, then there exists an integer $\H_t(m,n;s,k)$.
\end{abstract}

\keywords{Relative Heffter array; multipartite complete graph; cyclic decomposition}
\subjclass[2010]{05B20; 05B30}

\maketitle

\section{Introduction}\label{sec:Intro}

Relative Heffter arrays are partially filled arrays (p.f. arrays for short)
introduced in \cite{RelH}, generalizing  the original idea of Dan Archdeacon, \cite{A}.
They are defined as follows.

\begin{defi}\label{def:RelativeH}
Let $v=2ms+t$ be a positive integer, where $t$ divides $2ms$, and let $J$ be the subgroup of order $t$
of the cyclic group $\Z_v$.
A $\H_t(m,n; s,k)$ \emph{Heffter array  over $\Z_v$ relative to $J$} is an $m\times n$ p.f.  array
with elements in $\Z_v$ such that:
\begin{itemize}
\item[($\rm{a})$] each row contains $s$ filled cells and each column contains $k$ filled cells;
\item[($\rm{b})$] for every $x\in \Z_{v}\setminus J$, either $x$ or $-x$ appears in the array;
\item[($\rm{c})$] the elements in every row and column sum to $0$.
\end{itemize}
\end{defi}

In the square case (i.e., when $m=n$ and so $s=k$), the array $\H_t(n,n; k,k)$  will be denoted by
$\H_t(n;k)$. Note that when  $t=1$, that is $J$ is the trivial subgroup of $\Z_{2ms+1}$, one retrieves the
classical concept of
Heffter array. `Classical' Heffter arrays have been studied in several papers,
mainly  because they allow to produce biembeddings of orthogonal cyclic cycle decompositions
of the complete graph $K_v$ on  $v$ vertices onto orientable surfaces (see \cite{BCDY,CDYBiem,CMPPHeffter,DM} and
\cite{Dalai}).
Analogously, exploiting their connection with relative difference families (see \cite{RelH} and 
\cite{BP,Tomm,Francescola}),
relative Heffter arrays can be used for constructing pairs of orthogonal cyclic decompositions
(one decomposition consisting of $s$-cycles and the other one consisting of $k$-cycles) of the complete multipartite 
graph $K_{\ell\times t}$ 
with $\ell=\frac{v}{t}$ parts, each of size $t$.
Also, under suitable conditions (in particular, one needs orderings of the cells that satisfy certain properties), 
one can obtain  biembeddings of these pairs of orthogonal cyclic 
decompositions of $K_{\ell \times t}$ onto orientable surfaces (see \cite{RelHBiem}).

A relative Heffter array is called \emph{integer} if Condition ($\rm{c}$) in Definition \ref{def:RelativeH} is
strengthened so that the elements in every row and in every column, viewed as integers in
$\pm\left\{ 1, \ldots, \left\lfloor \frac{v}{2}\right\rfloor \right\}$,
sum to zero in $\Z$.
The \emph{support} of an integer Heffter array $A$, denoted by $\supp(A)$, is defined to be the set of the absolute
values of the elements contained in $A$.
It easily follows that an integer $\H_2(m,n;s,k)$ is nothing but an integer $\H_1(m,n;s,k)$,
since in both cases the support is the set $\{1,2,\ldots,ms\}$.

The existence problem for square Heffter arrays $\H_1(n;k)$ has been considered and solved in a series of
recent papers. In particular, by \cite{ADDY, DW} integer Heffter arrays $\H_1(n;k)$ (and integer $\H_2(n;k)$) exist if 
and only if $n\geq k\geq 3$ and $nk\equiv 0,3\pmod 4$.
Dropping the integer assumption, in \cite{CDDY} it was proved that Heffter arrays $\H_1(n;k)$ exist for all $n\geq 
k\geq 
3$.
Regarding  non-square arrays, only the tight case has been solved: in \cite{ABD} it was proved that
there exists a $\H_1(m,n;n,m)$ for all $m,n\geq 3$, and there exists
an integer $\H_1(m,n;n,m)$ if and only if the additional condition $mn\equiv 0,3 \pmod 4$ holds.

Necessary conditions for the existence of an integer $\H_t(m,n;s,k)$, in addition to the trivial ones
$3\leq s \leq n$, $3\leq k \leq m$ and $ms=nk$, are given by the following result which can be easily deduced
from  \cite[Proposition 3.1]{RelH}, reapplying the original argument  on the columns.

\begin{prop}\label{prop:necc}
Suppose there exists an integer $\H_t(m,n;s,k)$ for some divisor $t$ of $2ms=2nk$.
\begin{itemize}
\item[(1)] If $t$ divides $ms$, then $ms\equiv 0 \pmod 4$ or $ms\equiv -t \equiv \pm 1\pmod 4$.
\item[(2)] If $t=2ms$, then $s$ and $k$ must be both even.
\item[(3)] If $t\neq 2ms$ does not divide $ms$, then $t+2ms\equiv 0 \pmod 8$.
\end{itemize}
\end{prop}

The previous results on integer $\H_1(n;k)$ state that for $t=1,2$ these necessary conditions are actually also
sufficient. The same holds also for $t=k\neq 5$, as proved in \cite{RelH} (the existence of an integer $\H_5(n;5)$ is 
still an open problem for $n\equiv 0\pmod 4$).
However, in the same paper the authors showed that there is no integer $\H_{3n}(n;3)$ and no integer $\H_8(4;3)$, even
if conditions of Proposition \ref{prop:necc} hold.
The existence of an integer $\H_{n}(n;3)$ and of an integer $\H_{2n}(n;3)$ was proved  in  \cite{RelHBiem} for all odd 
$n\geq 3$.
No other case has been studied so far: this leaves the existence problem of an integer $\H_t(m,n;s,k)$ widely 
open.

In this paper we consider the case when $s$ and $k$ are both even.
In particular, we show that the previous necessary conditions are also sufficient when $s$ and $k$ are even and 
$ms\equiv 0 \pmod 4$. 
In other words, we prove the following result.

\begin{thm}\label{main}
Let $m,n,s,k$ be  integers such that $4\leq s\leq n$, $4\leq k \leq m$ and $ms=nk$.
Let $t$ be a divisor of $2ms$. 
\begin{itemize}
\item[(1)] If $s,k \equiv 0 \pmod 4$, then there exists an integer $\H_t(m,n;s,k)$.
\item[(2)] If $s\equiv 2\pmod 4$ and $k\equiv 0 \pmod 4$, then there exists an integer $\H_t(m,n;s,k)$ if and only if 
$m$ is even.
\item[(3)] If $s\equiv 0\pmod 4$ and $k\equiv 2 \pmod 4$, then there exists an integer $\H_t(m,n;s,k)$ if and only if 
$n$ is even.
\item[(4)] Suppose that $m$ and $n$ are both even. If $s,k\equiv 2 \pmod 4$, then there exists an integer 
$\H_t(m,n;s,k)$.
\end{itemize}
\end{thm}

This result is proved constructively  in Sections \ref{s0k0}, \ref{s2k0} and \ref{s2k2}: 
item (1) follows from Proposition \ref{prop:k4}; 
items (2) and (3) follow from Proposition \ref{s2};
item (4) follows from Proposition \ref{sk2}.
Unfortunately, the case when $m$ and $n$ are both odd and $s,k \equiv 2 \pmod 4$ remains open.
Note that, under these hypotheses, $t$ cannot be a divisor of $ms$.

To conclude, we point out that from Theorem \ref{main}, \cite[Proposition 2.9]{RelH} and  \cite[Theorem 4.1]{CMPPSums}
we obtain the following result concerning
cyclic cycle decompositions of $K_{\frac{2ms+t}{t}\times t}$.

\begin{cor}
Let $m,n,s,k$ be positive  such that $4\leq s\leq n$, $4\leq k \leq m$ and $ms=nk$.
Let $t$ be a divisor of $2ms$.
There exists a pair $(\mathcal{D}_1, \mathcal{D}_2)$ of orthogonal cyclic  decompositions of the graph 
$K_{\frac{2ms+t}{t}\times t}$, where $\mathcal{D}_1$ consists of $s$-cycles and $\mathcal{D}_2$ consists of $k$-cycles, 
in each of the following cases:
\begin{itemize}
 \item[(1)] $(s,k)\in \{(4,4), (4,6), (4,8), (6,4), (6,8),  (8,4), (8,6), (8,8)\}$;
 \item[(2)] $(s,k)=(6,6)$, with $m$ and $n$ both even.
\end{itemize}
\end{cor}

\noindent This result can be extended to any even $s$ and $k$ assuming the validity of \cite[Conjecture 3]{CMPPSums}.

\section*{Acknowledgments}

The second author was supported by the National Group for Algebraic and Geometric Structures, and their 
Applications (GNSAGA--INdAM).

\section{Notations}

In this paper,  the arithmetic on the row (respectively, on the column) indices is performed modulo $m$
(respectively, modulo $n$), where the set of reduced residues is $\{1,2,\ldots,m\}$ (respectively,
 $\{1,2,\ldots,n\}$), while the entries of the arrays are taken in $\Z$.
Given two integers $a\leq b$, we denote by $[a,b]$ the interval containing the integers $\{a,a+1,\ldots,b\}$.
If $a>b$, then $[a,b]$ is empty. We denote by $(i,j)$ the cell in the $i$-th row and $j$-th column 
of an array $A$. The skeleton of $A$ is the set of its filled positions.

If $A$ is an $m\times n$ p.f. array, for $i\in[1,n]$ we define the $i$-th diagonal as
$$D_i=\{(1,i),(2,i+1),\ldots,(m,i+m-1)\}.$$

\begin{defi}
A p.f. array with entries in $\Z$ is said to be \emph{shiftable} if
every row and every column contains an equal number of positive and negative entries.
\end{defi}

Let $A$ be a shiftable p.f. array and $x$ be a nonnegative integer.
Let $A\pm x$ be the (shiftable) p.f. array obtained adding $x$ to each positive
entry of $A$ and $-x$   to each negative entry of $A$. 
Observe that, since  $A$ is shiftable, the row and column sums of $A\pm x$ are exactly
the row and column sums of $A$.

Given a sequence $S=(B_1,B_2,\ldots)$ of shiftable p.f. arrays and a nonnegative integer $x$, 
we write $S\pm x$ for the sequence $(B_1\pm x, B_2\pm x,\ldots)$.

We denote by $\tau_i(A)$ and $\gamma_j(A)$ the sum of the elements of the $i$-th row and the sum of the 
elements of the $j$-th column, respectively, of  a p.f. array $A$.

If $S_1=(a_1,a_2,\ldots,a_{r})$ and $S_2=(b_1,b_2,\ldots,b_{u})$ are two sequences,
by $S_1 \con S_2$ we mean 
the sequence $(a_1,a_2,\ldots,a_r, b_1, b_2, \ldots, b_u)$ obtained by concatenation
of $S_1$ and $S_2$. In particular, if $S_1$ is the empty sequence then $S_1 \con S_2=S_2$.
Furthermore, given the sequences $S_1,\ldots,S_c$, we write
$\con\limits_{i=1}^c S_i$  for $(\cdots((S_1 \con S_2) \con S_3) \con  \cdots) \con  S_c$.

Finally, we recall that if $A$ is an integer $\H_t(m,n;s,k)$, then
$$
\supp(A)=\left[ 1,ms+\left\lfloor \frac{t}{2} \right\rfloor \right]
\setminus \left\{\ell, 2\ell, \ldots,  \left\lfloor \frac{t}{2} \right\rfloor \ell\right\},
\quad \textrm{ where } \ell=\frac{2ms}{t}+1.$$

\section{The case $s,k\equiv 0 \pmod 4$}\label{s0k0}

In this section we prove the existence of an integer 
$\H_t(m,n;s,k)$ when both $s$ and $k$ are divisible by $4$.
First of all, we set
$$d=\gcd(m,n),\quad m=d\bar m,\quad n= d \bar n,\quad s=4 \bar s \equad k=4\bar k.$$
Note that from $ms=nk$ we obtain that $\bar n$ divides $\bar s$ and $\bar m$ divides $\bar k$. Hence, we can write $\bar 
s= c \bar n$
and $\bar k=c \bar m$.
 
Fix two integers $a,b\geq 2$ and consider the following shiftable p.f. array:
$$B=B_{a,b}=\begin{array}{|c|c|}\hline
     1 & -(a+1) \\\hline
      & \\\hline
   -(b+1) & a+b+1\\\hline
    \end{array}\;.$$
Note that the sequences of the  row/column sums are $(-a,a)$ and $(-b,b)$, respectively.
We  use this $3\times 2$ block for constructing  p.f. arrays whose rows and columns sum to zero.
Start taking an empty $m\times n$ array $A$, fix a set $X$ of
$m\bar n$ nonnegative integers $x_0, x_1,\ldots,x_{m\bar n-1}$, and arrange the blocks
$B\pm x_j$ in such a way that the element $1+x_j$ fills the cell
$(j+1,j+1)$ of $A$ (recall that we work modulo $m$ on row indices  and modulo $n$ on column indices).
In this way, we fill the diagonals $D_{im-1}, D_{im}, D_{im+1}, D_{im+2}$ with $i\in [1,\bar n]$.
In particular, every row has $4\bar n$ filled cells and every column has $4\bar m$ filled cells.
 
Looking at the rows, the elements belonging to the
diagonals $D_{im+1},D_{im+2}$ sum to $-a$, while the elements belonging to the diagonals $D_{im-1},D_{im}$ sum to $a$.
Looking at the columns, the elements belonging to the diagonals $D_{im+1},D_{im-1}$
sum to $-b$, while the elements belonging to the diagonals $D_{im+2},D_{im}$
sum to $b$. Then $A$ has row/column sums equal to zero.

Applying this process $c$ times (working with the diagonals $D_{im+3},D_{im+4}, D_{im+5},D_{im+6}$, and so on),
we obtain a p.f. array $A$, whose rows have exactly $4\bar n \cdot c=s$ filled cells
and whose columns have exactly $4\bar m \cdot c =k$ filled cells.

\begin{ex}
For $a=2$ and $b=5$, fixing $X=\{0,1, 10, 11, 20, 21, 30, 31, 40, 41, 50, 51 \}$, we can fill the diagonals
$D_1,D_2,D_5, D_6, D_7, D_8, D_{11}, D_{12}$ of the following $6\times 12$ p.f. array, 
where we highlighted the block $B_{2,5}$:

\begin{footnotesize}
$$A=\begin{array}{|c|c|c|c|c|c|c|c|c|c|c|c|}\hline
\omb 1  & \omb -3 &      &       &  -26  &   28 &  31  &  -33 &     &      &  -56 & 58   \\\hline
      59 &        2 &   -4 &       &       &  -27 &   29 &  32  & -34 &      &      & -57    \\\hline
\omb -6 &  \omb 8 &   11 &   -13 &       &      & -36  &  38  & 41  &  -43 &      &     \\\hline
         &       -7 &   9  &    12 &  -14  &      &      & -37  & 39  &  42  & -44  &     \\\hline
         &          &  -16 &   18  &    21 &  -23 &      &      & -46 &  48  &  51  & -53 \\\hline
 -54     &          &      &   -17 &    19 &   22 &  -24 &      &     &  -47 &  49  &  52    \\\hline
  \end{array}\;.$$
\end{footnotesize}

\noindent
Note that $\supp(A)=[1,60]\setminus\{5j: j \in [1, 12]\}$. 
As the reader can verify, $A$ is an integer $\H_{24}(6,12;8,4)$: 
in this case $\ell=\frac{2\cdot 6\cdot 8}{24}+1=5$.
\end{ex}

The three constructions we present in this section are obtained following this procedure,
so they all produce p.f. arrays of size $m\times n$ whose rows and columns sum to zero.
To obtain an integer $\H_t(m,n;s,k)$ with $s,k\equiv 0 \pmod 4$, we only have to determine two integers $a,b\geq 2$
and a set $X=\left\{x_0,x_1,\ldots,x_{ms/4-1}\right\}\subset \N$ such that the p.f. array  constructed
using the blocks $B_{a,b}\pm x_j$ has the right support.
For instance, we can arrange the blocks in such a way that the element $1+x_j$ fills
the cell $(j+1,4q_j+j+1)$, where $q_j$ is the quotient of the division of $j$ by $\lcm(m,n)$.

Throughout this section we always assume that 
$4\leq s \leq n$, $4\leq k \leq m$, $ms=nk$ and $s,k\equiv 0 \pmod 4$.

\begin{lem}\label{k4-3}
There  exists an integer $\H_t(m,n;s,k)$ for any divisor $t$ of $2ms$ such that $t\equiv 0 \pmod 8$.
\end{lem}

\begin{proof}
Let  $B=B_{\ell,2\ell}=\begin{array}{|c|c|}\hline
     1 & -(\ell+1) \\\hline
      & \\\hline
   -(2\ell+1) & 3\ell +1\\\hline
    \end{array}\;$, where $\ell=\frac{2ms}{t}+1$.
An integer $\H_t(m,n;s,k)$, say $A$, can be obtained following the construction described before, once we exhibit a 
suitable set $X$
of size $\frac{ms}{4}$, in such a way that $\supp(A)=\left[1,ms+\frac{t}{2}\right]\setminus  \left\{\ell, 2\ell, 
\ldots, \frac{t}{2}\ell\right\}$.

Start considering the set $X_0=[0,\ell-2]$ of size $\ell-1=\frac{2ms}{t}$: it is easy to see that
$\bigcup\limits_{x \in X_0} \supp(B\pm x)=[1,4\ell]\setminus \{\ell, 2\ell, 3\ell, 4\ell \}$.
Similarly, for any $i \in \N$,  if $X_i=[4i\ell, (4i+1)\ell-2 ]$, then
$$\bigcup_{x \in X_i} \supp(B\pm x)=[4i\ell+1, (4i+4)\ell ]
\setminus \{(4i+1)\ell, (4i+2)\ell, (4i+3)\ell, (4i+4)\ell \}.$$
Clearly, $X_{i_1}\cap X_{i_2}=\emptyset$ if $i_1\neq i_2$.
So, take $X=\bigcup\limits_{i=0}^{t/8-1} X_i$: this is a set of size $\frac{t}{8}\cdot \frac{2ms}{t}=\frac{ms}{4}$, as 
required.
Also, the p.f. array $A$ obtained using the blocks $B\pm x$ with $x\in X$ has support
equal to
$$\begin{array}{rcl}
\supp(A) & =& \bigcup\limits_{i=0}^{t/8-1 } \left([4i\ell+1, (4i+4)\ell ] \setminus
 \{(4i+1)\ell, (4i+2)\ell, (4i+3)\ell, (4i+4)\ell \}\right)\\[8pt]
 & = & \left[1,\frac{t}{2}\ell\right]\setminus \left\{\ell, 2\ell, \ldots, \frac{t}{2}\ell\right\}
=\left[1,ms+\frac{t}{2}\right]\setminus  \left\{\ell, 2\ell, \ldots, \frac{t}{2}\ell\right\}.
\end{array}$$
This shows that $A$ is an integer $\H_t(m,n;s,k)$.
\end{proof}

For instance, to construct an integer $\H_{16}(5,10;8,4)$ we can follow the proof of the previous lemma.
In fact, $t=16$ divides $2\cdot 5\cdot 8$; note that $\ell=6$.

\begin{footnotesize}
$$\H_{16}(5,10;8,4)=\begin{array}{|c|c|c|c|c|c|c|c|c|c|}\hline
\omb    1&  \omb  -7&     &   -16&    22&    25&   -31&     &   -40&    46  \\\hline
   47&     2&    -8&     &   -17&    23&    26&   -32&     &   -41 \\\hline
\omb  -13&  \omb  19&     3&    -9&     &   -37&    43&    27&   -33&     \\\hline
    &   -14&    20&     4&   -10&     &   -38&    44&    28&   -34 \\\hline
  -35&     &   -15&    21&     5&   -11&     &   -39&    45&    29 \\\hline
  \end{array}\;.$$
\end{footnotesize}

\begin{lem}\label{k4-2}
There exists an integer $\H_t(m,n;s,k)$ for any divisor $t$ of
$ms$ such that $t\equiv 0 \pmod 4$.
\end{lem}

\begin{proof}
Let  $B=B_{1,\ell}=\begin{array}{|c|c|}\hline
     1 & -2 \\\hline
      & \\\hline
   -(\ell+1) & \ell+2\\\hline
    \end{array}\;$ and note that, since $t$ divides $ms$, $\ell=\frac{2ms}{t}+1$ is an odd integer.
We start considering the set $X_0=\{0,2,4,\ldots,\ell-3\}$ of size
$\frac{\ell-1}{2}=\frac{ms}{t}$: it is easy to see that
$\bigcup\limits_{x \in X_0} \supp(B\pm x)=[1,\ell-1]\cup [\ell+1,2\ell-1]=[1,2\ell]\setminus
\{\ell, 2\ell\}$.
Similarly, for any $i \in \N$,  if $X_i=\{2i\ell,2i\ell+2,2i\ell+4,\ldots, (2i+1)\ell-3 \}$, then
$$\bigcup_{x \in X_i} \supp(B\pm x)=[2i\ell+1, 2(i+1)\ell ]\setminus
\{(2i+1)\ell, (2i+2)\ell \}$$
and $X_{i_1}\cap X_{i_2}=\emptyset$ if $i_1\neq i_2$.
So, take $X=\bigcup\limits_{i=0}^{t/4-1} X_i$: this is a set of size
$\frac{t}{4}\cdot \frac{ms}{t}=\frac{ms}{4}$, as required.
Hence, the p.f. array $A$ obtained following our procedure and using the blocks $B\pm x$ with $x\in X$ has support
equal to
$$\begin{array}{rcl}
\supp(A) & =& \bigcup\limits_{i=0}^{t/4-1}
\left([2i\ell+1, 2(i+1)\ell ]\setminus
\{(2i+1)\ell, (2i+2)\ell \} \right)\\[8pt]
 & = & \left[1,\frac{t}{2}\ell\right]\setminus \left\{\ell, 2\ell, \ldots, \frac{t}{2}\ell\right\}
=\left[1,ms+\frac{t}{2}\right]\setminus  \left\{\ell, 2\ell, \ldots, \frac{t}{2}\ell\right\}.
\end{array}$$
It follows that $A$ is an integer $\H_t(m,n;s,k)$.
\end{proof}

Following the proof of the previous lemma, we can construct an integer $\H_{12}(9;8)$.
In fact, $t=12$ divides $9\cdot 8$; note that $\ell=13$.

\begin{footnotesize}
$$\H_{12}(9;8)=\begin{array}{|c|c|c|c|c|c|c|c|c|}\hline
\omb 1 & \omb -2 & -74  & 75  &33  &-34  &  &-42  &43  \\\hline
45 & 3 & -4 & -76& 77& 35& -36& & -44 \\\hline
\omb -14 & \omb 15 & 5  & -6  &-46 & 47 & 37 &-38  &    \\\hline
& -16& 17& 7& -8& -48& 49& 53& -54 \\\hline
 -56& & -18& 19& 9& -10& -50& 51& 55 \\\hline
57& -58& & -20& 21& 11& -12& -66& 67 \\\hline
69& 59& -60& & -22& 23& 27& -28& -68 \\\hline
-70& 71& 61& -62& & -24& 25& 29& -30 \\\hline
-32& -72& 73& 63& -64& & -40& 41& 31 \\\hline
  \end{array}\;.$$
\end{footnotesize}

\begin{lem}\label{k4-1}
There exists an integer $\H_t(m,n;s,k)$ for any divisor $t$ of $\frac{ms}{2}$.
\end{lem}

\begin{proof}
Let  $B=B_{1,2}=\begin{array}{|c|c|}\hline
     1 & -2 \\\hline
      & \\\hline
   -3 & 4\\\hline
    \end{array}\;$. Note that  $\ell=\frac{2ms}{t}+1\equiv 1 \pmod 4$ since $t$ divides $\frac{ms}{2}$.
We start considering the set $X_0=\{0,4,8,\ldots,\ell-5 \}$ of size
$\frac{\ell-1}{4}=\frac{ms}{2t}$: it is easy to see that
$\bigcup\limits_{x \in X_0} \supp(B\pm x)=[1,\ell]\setminus
\{\ell\}$.
Similarly, for any $i \in \N$,  if $X_i=\{i\ell,i\ell+4,i\ell+8,\ldots, (i+1)\ell-5\}$, then
$$\bigcup_{x \in X_i} \supp(B\pm x)=[i\ell+1, (i+1)\ell ]\setminus
\{(i+1)\ell \}$$
and $X_{i_1}\cap X_{i_2}=\emptyset$ if $i_1\neq i_2$.

If $t$ is even, take $X=\bigcup\limits_{i=0}^{t/2-1} X_i$: this is a set of size
$\frac{t}{2}\cdot \frac{ms}{2t}=\frac{ms}{4}$, as required.
The p.f. array $A$ obtained following our procedure and using the blocks $B\pm x$ with $x\in X$ has support
equal to
$$\begin{array}{rcl}
\supp(A) & =& \bigcup\limits_{i=0}^{t/2-1}
\left( [i\ell+1, (i+1)\ell ]\setminus
\{(i+1)\ell \} \right)\\[8pt]
 & = & \left[1,\frac{t}{2}\ell\right]\setminus \left\{\ell, 2\ell, \ldots, \frac{t}{2}\ell\right\}
=\left[1,ms+\frac{t}{2}\right]\setminus  \left\{\ell, 2\ell, \ldots, \frac{t}{2}\ell\right\}.
\end{array}$$
Suppose now that $t$ is odd. Notice that, in this case, $\ell\equiv 1\pmod 8$. Take
$$Y=\left\{\left(\frac{t-1}{2}\right)\ell, \left(\frac{t-1}{2}\right)\ell+4, \left(\frac{t-1}{2}\right)\ell+8,\ldots,
\left(\frac{t-1}{2}\right)\ell+ \frac{\ell-9}{2} \right\}.$$
Then  $|Y|=\frac{ms}{4t}$ and
$\bigcup\limits_{y \in Y} \supp(B\pm y)=\left[\left(\frac{t-1}{2}\right)\ell+1, 
\left(\frac{t-1}{2}\right)\ell+\frac{\ell-1}{2} \right]$.
Take $X=\left(\bigcup\limits_{i=0}^{(t-3)/2} X_i\right)\cup Y$: this is a set of size
$\frac{t-1}{2}\cdot \frac{ms}{2t}+\frac{ms}{4t}=\frac{ms}{4}$,
as required.
In this case, the p.f. array $A$ obtained following our procedure and using the blocks $B\pm x$ with $x\in X$ has 
support:
$$\begin{array}{rcl}
\supp(A) & =& \bigcup\limits_{i=0}^{\frac{t-3}{2}}
\left( [i\ell+1, (i+1)\ell ]\setminus
\{(i+1)\ell \} \right)\cup \left(\left[\left(\frac{t-1}{2}\right)\ell+1, 
\left(\frac{t-1}{2}\right)\ell+\frac{\ell-1}{2} 
\right] \right)\\[8pt]
 & = & \left(\left[1, \frac{t-1}{2}\ell \right]\setminus \left\{\ell, 2\ell, \ldots, 
\frac{t-1}{2}\ell\right\}\right)\cup
\left[\left(\frac{t-1}{2}\right)\ell+1, ms+\frac{t-1}{2} \right]
\\[8pt]
& =& \left[1,ms+\left\lfloor\frac{t}{2}\right\rfloor\right]\setminus
\left\{\ell, 2\ell, \ldots, \left\lfloor\frac{t}{2}\right\rfloor\ell\right\}.
\end{array}$$

In both cases, we obtain that $A$ is an integer $\H_t(m,n;s,k)$.
\end{proof}

For instance, we can follow the proof of the previous lemma for constructing an integer $\H_{10}(5,10;8,4)$.
In fact, $t=10$ divides $\frac{5\cdot 8}{2}=20$. Note that $\ell= 9$.

\begin{footnotesize}
$$\H_{10}(5,10;8,4)=\begin{array}{|c|c|c|c|c|c|c|c|c|c|}\hline
 \omb   1&  \omb  -2&     &   -16&    17&    23&   -24&    &   -39&    40  \\\hline
     44&     5&    -6&     &   -21&    22&    28&   -29&     &   -43 \\\hline
 \omb   -3&  \omb   4&    10&   -11&    &   -25&    26&    32&   -33&      \\\hline
      &    -7&     8&    14&   -15&     &   -30&    31&    37&   -38 \\\hline
    -42&     &   -12&    13&    19&   -20&     &   -34&    35&    41 \\\hline 
  \end{array}\;.$$
\end{footnotesize}

\begin{prop}\label{prop:k4}
Suppose $4\leq s \leq n$, $4\leq k \leq m$, $ms=nk$ and $s,k\equiv 0 \pmod 4$.
Then, there exists a shiftable integer $\H_t(m,n;s,k)$ for every divisor $t$ of $2ms$.
\end{prop}

\begin{proof}
Let $t$ be a divisor of $2ms$.
If $t \equiv 0 \pmod 8$, then we apply Lemma \ref{k4-3}.
If $t \equiv 4 \pmod 8$, then $t$ divides $ms$ and hence we can apply Lemma \ref{k4-2}.
Finally, if $t\not \equiv 0 \pmod 4$, then $t$ divides $\frac{ms}{2}$ and so the existence of an  integer
$\H_t(m,n;s,k)$  follows from Lemma \ref{k4-1}.
Note that in all these three cases, the integer relative Heffter array that we construct is
shiftable.
\end{proof}

\section{The case  $s\equiv 2 \pmod 4$ and $k\equiv 0\pmod 4$}\label{s2k0}

In this section, we will assume $s\geq 6$ and $k\geq 4$ with $s \equiv 2 \pmod 4$ and $k\equiv 0 \pmod 4$.
From $ms=nk$ it follows that $m$ must be even.
To prove the existence of an integer $\H_t(m,n;s,k)$, we start determining the skeleton of such arrays. To this 
purpose, we define a `base unit' that we will fill with the elements of suitable blocks.

Let $\B$ be a sequence of blocks such that the following property is satisfied:
\begin{equation}\label{blocchi2}
\begin{array}{l}
\textrm{fixed } b \textrm{ integers } \sigma_1,\ldots,\sigma_b, \textrm{ the elements of } \B  
\textrm{ are  blocks } B \textrm{ of size } 2\times 2b \\
\textrm{such that } \gamma_{2i-1}(B)=-\gamma_{2i}(B)=\sigma_i 
\textrm{ for all } i \in [1,b].
\end{array}
\end{equation}
For instance, the following block satisfies condition \eqref{blocchi2} with $\sigma_1=-3$, $\sigma_2=-2$ and 
$\sigma_3=1$:
$$\begin{array}{|r|r|r|r|r|r|}\hline
1 & -2 & 6 & -7 & -10 & 12 \\\hline
-4 & 5 & -8 & 9 & 11 & -13\\\hline
  \end{array}.$$
  
So, fix two integers $a$ and $d$ such that $1\leq 2a\leq d$. 
Let $\B=(B_1,\ldots,B_d)$ be a sequence satisfying \eqref{blocchi2}, where the blocks $B_{r}=(b_{i,j}^{(r)})$
are all of size $2\times 2a$.
Let $P=P(\B)$ be the p.f. array of size $2d\times d$ so defined.\label{bloccoP}
For all $i \in [1,a]$ and all $j \in [1,2a]$, 
the cell $(i,  i+j-1)$ of $P$ is filled with the element $b_{1,j}^{(i)}$ and the cell $(d+i,i+j-1)$ is filled with 
the element $b_{2,j}^{(i)}$;
here, the  column indices are taken modulo $d$.
The remaining cells of $P$ are empty. An example of such construction is given in Figure \ref{P}.

\begin{figure}[ht]
\begin{footnotesize}
$$\begin{array}{|c|c|c|c|c|c|}\hline
 b_{1,1}^{(1)} & b_{1,2}^{(1)} & b_{1,3}^{(1)} & b_{1,4}^{(1)} &   &    \\\hline
    &  b_{1,1}^{(2)} & b_{1,2}^{(2)} & b_{1,3}^{(2)} & b_{1,4}^{(2)} &   \\ \hline
    &   & b_{1,1}^{(3)} & b_{1,2}^{(3)} & b_{1,3}^{(3)} & b_{1,4}^{(3)}   \\ \hline
 b_{1,4}^{(4)}   &   &   & b_{1,1}^{(4)} & b_{1,2}^{(4)} & b_{1,3}^{(4)}  \\ \hline
 b_{1,3}^{(5)}  & b_{1,4}^{(5)}  &   &   & b_{1,1}^{(5)} & b_{1,2}^{(5)}  \\ \hline
 b_{1,2}^{(6)} & b_{1,3}^{(6)} & b_{1,4}^{(6)} &  & & b_{1,1}^{(6)}\\\hline
 b_{2,1}^{(1)} & b_{2,2}^{(1)} & b_{2,3}^{(1)} & b_{2,4}^{(1)} &   &    \\\hline
    &  b_{2,1}^{(2)} & b_{2,2}^{(2)} & b_{2,3}^{(2)} & b_{2,4}^{(2)} &   \\ \hline
    &   & b_{2,1}^{(3)} & b_{2,2}^{(3)} & b_{2,3}^{(3)} & b_{2,4}^{(3)}   \\ \hline
 b_{2,4}^{(4)}   &   &   & b_{2,1}^{(4)} & b_{2,2}^{(4)} & b_{2,3}^{(4)}  \\ \hline
 b_{2,3}^{(5)}  & b_{2,4}^{(5)}  &   &   & b_{2,1}^{(5)} & b_{2,2}^{(5)}  \\ \hline
 b_{2,2}^{(6)} & b_{2,3}^{(6)} & b_{2,4}^{(6)} &  & & b_{2,1}^{(6)}\\\hline
  \end{array}$$
  \end{footnotesize}
  \caption{This is a $P(B_1,\ldots,B_6)$, where $B_1,\ldots,B_6$ are arrays of size  $2\times 4$.}\label{P}
\end{figure}

We prove that $P$ is a p.f. array whose columns all sum to zero.
Observe that every row of $P$ contains exactly $2a$ filled cells and every column contains exactly $4a$ elements.
The elements of the  $i$-th column of $P$ are
$$b_{1,1}^{(i)}, b_{1,2}^{(i-1)}, \ldots, b_{1,2a}^{(i+1-2a)},\; 
b_{2,1}^{(i)}, b_{2,2}^{(i-1)}, \ldots, b_{2,2a}^{(i+1-2a)},$$
where the exponents must be read modulo $d$, with residues in $[1,d]$.
Since the sequence $\B$ satisfies \eqref{blocchi2}, we obtain
$$\gamma_i(P)=  \sum_{j=1}^{2a} \gamma_j (B_{i+1-j}) =
\sum_{j=1}^{2a} \gamma_j ( B_{i}) =
\sum_{u=1}^a (\sigma_u-\sigma_u)=0.$$
Furthermore, notice that $\tau_j(P)=\tau_1(B_j)$ and $\tau_{d+j}(P)=\tau_2(B_j)$ for all $j \in [1,d]$.

Our strategy consists of two steps: first of all, we will show that there are suitable sequences of 
blocks satisfying condition \eqref{blocchi2};
then, we will use these sequences to obtain an integer $\H_t(m,n;s;k)$ using p.f. arrays  of type $P(\B)$.

We actually construct sequences $\B$  satisfying this stronger condition:
\begin{equation}\label{blocchi}
\begin{array}{l}
\textrm{fixed } b \textrm{ integers } \sigma_1,\ldots,\sigma_b, \textrm{ the elements of } \B  \textrm{ are  
shiftable blocks } B \textrm{ of size}\\
2\times 2b \textrm{ such that }\tau_1(B)=\tau_2(B)=0 \textrm{ and }
\gamma_{2i-1}(B)=-\gamma_{2i}(B)=\sigma_i \textrm{ for all } i \in [1,b].
\end{array}
\end{equation}
This condition includes the former \eqref{blocchi2} and  will help us to control the row sums.

We start considering the case when $\ell=\frac{2ms}{t}+1$ is even.

\begin{lem}\label{8p}
Let $m$ and $s$ be even integers with $m\geq 2$ and $s\geq 6$
such that there exists an odd prime $p$ dividing $s$.
Let $t$ be a divisor of $2ms$ such that $t\equiv 0 \pmod {8p}$.
There exists a sequence $\B$ of $\frac{m}{2}$ shiftable blocks of size $2\times s$ such  that $\B$ satisfies 
condition \eqref{blocchi} and
$\supp(\B)=[ 1,ms+ t/2 ] \setminus \{j\ell: j \in [1,  t/2 ]\}$.
\end{lem}

\begin{proof}
Take the blocks of Figure \ref{blW1}.
 Then $W_4$ and $W_6$ satisfy property \eqref{blocchi} with column sums
 $(-2\ell, 2\ell, 2\ell, -2\ell)$ and 
 $(-2\ell, 2\ell, -2\ell,$ $2\ell, \ell,-\ell)$, respectively.
 Furthermore,
 $$\supp(W_4)=\{j\ell+1: j \in [0,7] \} \equad \supp(W_6)=\{j\ell+1: j \in [0,11]\}.$$
 Let $V$ be the following $2\times 2p$ block:
 $$V= \begin{array}{|c|c|c|c|c|}\hline
 W_6  & W_4\pm 12\ell & W_4\pm 20\ell & \cdots & W_4\pm (4p-8)\ell \\\hline
 \end{array}\;.
 $$
 Clearly, also $V$ satisfies \eqref{blocchi} and its support is
 $\supp(V)=\{j\ell+1: j \in [0,4p-1] \}$.
We can use this block $V$ for constructing our sequence $\B$:
the $2\times s$ blocks of $\B$ are obtained simply by juxtaposing $h=\frac{s}{2p}$
blocks of type $V\pm x$, for $x\in X \subset \N$, following the natural order of $(X,\leq)$.  
 So, we are left to exhibit  a suitable set $X$ of size $\frac{mh}{2}$
such that the support of the corresponding sequence $\B$ is 
$[1,ms+t/2]\setminus \{j\ell: j \in [1,t/2]\}$.

 Let first $X_0=[0,\ell-2]$. Then $\supp(V\pm x_{i_1})\cap \supp(V\pm x_{i_2})=\emptyset$ for each $x_{i_1},x_{i_2}\in
X_0$ such that $x_{i_1}\neq x_{i_2}$. Furthermore,
 $$\bigcup_{x \in X_0} \supp(V\pm x)=[1,4p\ell]\setminus \{j\ell: j \in [1,4p] \}.$$
 Similarly, for any $i\in \N$, if  $X_i=[4pi\ell, (4pi+1)\ell-2]$ then
 $$\bigcup_{x \in X_i} \supp(V\pm x)=[1+4pi\ell, 4p\ell+4pi\ell]\setminus \{j\ell: j \in [1+4pi,4p+4pi] \}.$$
 Clearly, $X_{i_1}\cap X_{i_2}=\emptyset$ if $i_1\neq i_2$.
 Therefore, take
 $X=\bigcup\limits_{i=0}^{\frac{t}{8p}-1} X_i$: this is a set of size $\frac{t}{8p}\cdot (\ell-1)=\frac{t}{8p}\cdot
 \frac{4mph}{t}=\frac{mh}{2}$.
 It follows that the sequence $\B$ obtained, as previously described, using the blocks $V\pm x$, with $x\in X$, has 
support equal to
 $$\begin{array}{rcl}
\supp(\B) & =&  \bigcup\limits_{i=0}^{\frac{t}{8p}-1}([1+4pi\ell,4p\ell+4pi\ell]\setminus \{j\ell: j \in [1+4pi,4p+4pi] 
\})\\[8pt]
 & =& \left [1,\frac{t}{2}\ell \right ]\setminus \left \{j\ell: j \in \left [1,\frac{t}{2}\right ] \right \}
 =\left[1,ms+\frac{t}{2}\right]\setminus  \left\{\ell, 2\ell, \ldots, \frac{t}{2}\ell\right\},
 \end{array}$$
 as required.
\end{proof}

\begin{figure}[ht]
\begin{footnotesize}
$$\begin{array}{rcl}
W_4 & =& \begin{array}{|c|c|c|c|}\hline
       1 &   -(\ell+1) &  -(4\ell+1) &    5\ell+1\\\hline
   -(2\ell+1) &  3\ell+1 &    6\ell+1 & -(7\ell+1)\\\hline
\end{array}\;,\\[8pt]
W_6 & = &
\begin{array}{|c|c|c|c|c|c|}\hline
   1 &         -(\ell+1) &   4\ell+1 &      -(5\ell+1) &    -(8\ell+1) &    10\ell+1 \\\hline
 -(2\ell+1) &  3\ell+1 &    -(6\ell+1) &     7\ell+1 &    9\ell+1 &    -(11\ell+1) \\\hline
\end{array}\;.
\end{array}$$
\end{footnotesize}
\caption{Shiftable blocks satisfying condition \eqref{blocchi}.}\label{blW1}
\end{figure}

\begin{ex}
Suppose $m=12$, $s=10$ and $t=80$. Following the proof of  Lemma \ref{8p}, where $p=5$, $h=1$, $\ell=4$
and $X=\{0,1,2\} \cup \{80,81,82 \}$, 
we obtain the following sequence $(B_1,\ldots,B_{6} )$ of shiftable blocks:

\begin{footnotesize}
$$\begin{array}{rcl}
B_1 & =& \begin{array}{|c|c|c|c|c| c|c|c|c|c|}\hline
   1& -5& 17& -21& -33& 41& 49& -53& -65& 69 \\ \hline 
    -9& 13& -25& 29& 37& -45& -57& 61& 73& -77 \\ \hline
    \end{array}\;, \\[8pt]
B_2 & =& \begin{array}{|c|c|c|c|c| c|c|c|c|c|}\hline
 2& -6& 18& -22& -34& 42& 50& -54& -66& 70 \\ \hline  
 -10& 14& -26& 30& 38& -46& -58& 62& 74& -78 \\ \hline 
    \end{array}\;, \\[8pt]
        \end{array}$$
 $$ \begin{array}{rcl}
    B_3 & =& \begin{array}{|c|c|c|c|c| c|c|c|c|c|}\hline
 3& -7& 19& -23& -35& 43& 51& -55& -67& 71 \\ \hline 
 -11& 15& -27& 31& 39& -47& -59& 63& 75& -79 \\ \hline 
    \end{array}\;, \\[8pt]
B_4 & =& \begin{array}{|c|c|c|c|c| c|c|c|c|c|}\hline
 81& -85& 97& -101& -113& 121& 129& -133& -145& 149 \\ \hline 
 -89& 93& -105& 109& 117& -125& -137& 141& 153& -157 \\ \hline 
    \end{array}\;, \\[8pt]
B_5 & =& \begin{array}{|c|c|c|c|c| c|c|c|c|c|}\hline
 82& -86& 98& -102& -114& 122& 130& -134& -146& 150 \\ \hline 
 -90& 94& -106& 110& 118& -126& -138& 142& 154& -158 \\ \hline 
    \end{array}\;, \\[8pt]
B_6 & =& \begin{array}{|c|c|c|c|c| c|c|c|c|c|}\hline
 83& -87& 99& -103& -115& 123& 131& -135& -147& 151 \\ \hline 
 -91& 95& -107& 111& 119& -127& -139& 143& 155& -159 \\ \hline 
  \end{array}\;.
\end{array}$$
\end{footnotesize}
\end{ex}

\begin{lem}\label{non 8p}
Let $m$ and $s$ be even integers with $m\geq 2$ and $s\geq 6$
such that there exists an odd prime $p$ dividing $s$.
Let $t$ be a divisor of $\frac{2ms}{p}$ such that $t\equiv 0 \pmod {8}$
and set $\ell=\frac{2ms}{t}+1$.
There exists a sequence $\B$ of $\frac{m}{2}$ blocks of size $2\times s$ such  that $\B$ satisfies 
condition \eqref{blocchi} and
$\supp(\B)=[ 1,ms+ t/2] \setminus \{j\ell: j \in [1, t/2 ]\}$.
\end{lem}

\begin{proof}
By hypothesis we can write $\ell=p y +1$. The blocks $W_4$ and $W_6$ of Figure \ref{blW2} satisfy property 
\eqref{blocchi} with 
column sums $(-y,y,y,-y)$
 and $(-2(py+1), 2(py+1), -2(py+1), 2(py+1), py+1, -(py+1) )$,  respectively.
 Furthermore,
 $$\begin{array}{rcl}
\supp(W_4) & =& \{ (jp+1)y+j+1  , (jp+2)y+j+1: j \in [0,3] \},\\
\supp(W_6) & =& \{ jpy+j+1, (jp+1)y+j+1, (jp+2)y+j+1: j \in [0,3]\}.
   \end{array}$$
Let $V$ be the following $2\times 2p$ block:
 $$V=
 \begin{array}{|c|c|c|c|c|}\hline
 W_6  & W_4\pm 2y & W_4\pm 4y & \cdots & W_4\pm (p-3)y \\\hline
 \end{array}\;.
 $$
 Clearly, also $V$ satisfies \eqref{blocchi} and its support is
 $$\begin{array}{rcl}
\supp(V) & =&  \{iy+1,(p+i)y+2,(2p+i)y+3, (3p+i)y+4 : i \in [0,p-1]\}\\
&= & \{iy+1,\ell+(iy+1),2\ell+(iy+1),3\ell+(iy+1):i\in[0,p-1]\}.
   \end{array}$$
 We can use this block $V$ for constructing our sequence $\B$ as done in Lemma \ref{8p}: 
 it suffices to exhibit  a suitable set $X$ of size 
$\frac{mh}{2}$,  where $h=\frac{s}{2p}$,   such that the support of the corresponding sequence  $\B$ is
$[1,ms+t/2]\setminus \{j \ell : j \in [1,t/2]\}$.

 Let first $X_0=[0,y-1]$. Then $\supp(V\pm x_{i_1})\cap \supp(V\pm x_{i_2})=\emptyset$ for each $x_{i_1},x_{i_2}\in
X_0$ such that $x_{i_1}\neq x_{i_2}$. Furthermore,
 $$\bigcup_{x \in X_0} \supp(V\pm x)= \bigcup\limits_{i=0}^3 [i\ell+1, i\ell+py] =[1,4\ell]\setminus \{\ell, 2\ell, 
3\ell, 4\ell\}.$$
 Similarly, for any $i\in \N$, if  $X_i=[4i\ell , 4i\ell +y -1]$ then
 $$\bigcup_{x \in X_i} \supp(V\pm x)=[1+4i\ell, (4i+4)\ell ]\setminus \{(4i+1)\ell, (4i+2)\ell, (4i+3)\ell, (4i+4)\ell  
\}.$$
 Clearly, $X_{i_1}\cap X_{i_2}=\emptyset$ if $i_1\neq i_2$.
 Therefore, take
 $X=\bigcup\limits_{i=0}^{\frac{t}{8}-1} X_i$: this is a set of size $\frac{t}{8}\cdot y=\frac{t}{8}\cdot \frac{\ell 
-1}{p}=
 \frac{t}{8} \cdot \frac{4mh}{t}=\frac{mh}{2}$.
 It follows that the sequence $\B$ obtained using the blocks $V\pm x$, with $x\in X$, has 
support equal to
 $$\begin{array}{rcl}
\supp(\B) & =&  \bigcup\limits_{i=0}^{\frac{t}{8}-1}([1+4i\ell,4\ell(i+1)]\setminus \{(4i+1)\ell, (4i+2)\ell, 
(4i+3)\ell, 
(4i+4)\ell  \})\\[8pt]
 & =& \left [1,\frac{t}{2}\ell \right ]\setminus \left \{\ell, 2\ell, \ldots, \frac{t}{2}\ell \right \}
 =\left[1,ms+\frac{t}{2}\right]\setminus  \left\{\ell, 2\ell, \ldots, \frac{t}{2}\ell\right\},
 \end{array}$$
 as required.
\end{proof}

 \begin{figure}[ht]
\begin{footnotesize} 
 $$\begin{array}{rcl}
W_4 & =& \begin{array}{|c|c|c|c|}\hline
   y+1 &    -((p+1)y+2) &  -((2p+1)y+3) &  (3p+1)y+4\\\hline
 -(2y+1) &    (p+2)y+2 &    (2p+2)y+3 &  -((3p+2)y+4) \\\hline
\end{array}\;,\\[8pt]
W' & = &
\begin{array}{|c|c|c|c|}\hline
      1 &  -(py+2) &  y+1  & -((p+1)y+2) \\\hline
 -(2py+3) &  3py+4 & -((2p+1)y+3)   &  (3p+1)y+4  \\\hline
\end{array}\;, \\[8pt]
W'' & = &
\begin{array}{|c|c|}\hline
 -(2y+1) & (2p+2)y+3)  \\\hline
   (p+2)y+2 &    -((3p+2)y+4)  \\\hline
\end{array}\;, \\[8pt]
W_6 & = &
\begin{array}{|c|c|}\hline
 W' & W''\\\hline
\end{array}\;.
\end{array}$$
\end{footnotesize}
\caption{Shiftable blocks $W_4, W_6$ satisfying condition \eqref{blocchi}.}\label{blW2}
\end{figure}

\begin{ex}
Suppose $m=12$, $s=10$ and $t=16$. Following the proof of  Lemma \ref{non 8p}, where $p=5$, $\ell=16$, $h=1$, $y=3$
and $X=\{0,1,2\} \cup \{64,65,66 \}$, 
we obtain the following sequence $(B_1,\ldots,B_{6} )$ of shiftable blocks:

\begin{footnotesize}
$$\begin{array}{rcl}
B_1 & =& \begin{array}{|c|c|c|c|c| c|c|c|c|c|}\hline
1& -17& 4& -20& -7& 39& 10& -26& -42& 58 \\\hline
 -33& 49& -36& 52& 23& -55& -13& 29& 45& -61  \\\hline
\end{array}\;, \\[8pt]
B_2 & =& \begin{array}{|c|c|c|c|c| c|c|c|c|c|}\hline
 2& -18& 5& -21& -8& 40& 11& -27& -43& 59  \\\hline
-34& 50& -37& 53& 24& -56& -14& 30& 46& -62  \\\hline
\end{array}\;, \\[8pt]
B_3 & =& \begin{array}{|c|c|c|c|c| c|c|c|c|c|}\hline
 3& -19& 6& -22& -9& 41& 12& -28& -44& 60  \\\hline
   -35& 51& -38& 54& 25& -57& -15& 31& 47& -63  \\\hline 
\end{array}\;, \\[8pt]
B_4 & =& \begin{array}{|c|c|c|c|c| c|c|c|c|c|}\hline
 65& -81& 68& -84& -71& 103& 74& -90& -106& 122  \\\hline
 -97& 113& -100& 116& 87& -119& -77& 93& 109& -125  \\\hline
\end{array}\;, \\[8pt]
B_5 & =& \begin{array}{|c|c|c|c|c| c|c|c|c|c|}\hline
 66& -82& 69& -85& -72& 104& 75& -91& -107& 123  \\\hline 
 -98& 114& -101& 117& 88& -120& -78& 94& 110& -126  \\\hline 
\end{array}\;, \\[8pt]
B_6 & =& \begin{array}{|c|c|c|c|c| c|c|c|c|c|}\hline
  67& -83& 70& -86& -73& 105& 76& -92& -108& 124  \\\hline
  -99& 115& -102& 118& 89& -121& -79& 95& 111& -127  \\\hline
\end{array}\;.
\end{array}$$
\end{footnotesize}
\end{ex}

We now consider the case when $\ell=\frac{2ms}{t}+1$ is odd (that is, $t$ divides $ms$).
We need to distinguish three possibilities, depending on the residue class of $s$ modulo $6$.
So,  write $s=6q+r$, where $q\geq 0$ and $r\in \{0,8,10\}$. 
We start with the following  auxiliary lemma, that allows us to apply an inductive process.

\begin{lem}\label{F_rho}
Let $h\geq 1$ and  let $\rho\geq 3$ be an odd integer. There exists a sequence $\F=\F_{h,\rho}$ consisting of
blocks of size $2\times 6$ such that
\begin{itemize}
\item[(1)] $\F$ has cardinality $h$ and satisfies condition \eqref{blocchi};
\item[(2)]
$\supp(\F)=\left[1, 12 h+\left\lfloor \frac{12h}{\rho-1} \right\rfloor \right]\setminus \left\{\rho, 2\rho, 
\ldots,\left\lfloor \frac{12h}{\rho-1} \right\rfloor \rho \right\}$.
\end{itemize}
\end{lem}

\begin{proof}
Consider the  $2\times 6$ shiftable blocks given in Figure \ref{bl6}.
Observe that $\supp(F_3)=[1,18]\setminus\{3,6,9,12,15,18\}$,
$\supp(F_5)=[1,15]\setminus\{5,10,15\}$ and $\supp(V_j)=[1,13]\setminus\{ j\}$.
Furthermore, each of these blocks has rows that sum to zero and  columns with the following sums:
$$F_3:\; (-1,1, -3,3, 6,-6),\qquad F_5, V_j=(-2,2,-2,2,1,-1).$$
We  construct our sequence $\F$ distinguishing three cases:
we use only blocks of type  $F_3\pm x$; or we use only blocks of type
$F_5\pm x$; or we use blocks of type  $V_j\pm x$ for several choices of $j$ in the same sequence.

\begin{figure}[ht]
\begin{footnotesize}
 $$\begin{array}{rclrcl}
F_{3} & =& \begin{array}{|r|r|r|r|r|r|}\hline
1 & -4 & -10 & 16 & 14 & -17\\\hline
-2 & 5 & 7 & -13 & -8 & 11 \\\hline
\end{array}\;, & 
F_{5} & =& \begin{array}{|r|r|r|r|r|r|}\hline
1 & -2 & 6 & -7 & -11 & 13\\\hline
-3 & 4 & -8 & 9 & 12 & -14 \\\hline
\end{array}\;,\\[8pt]
V_1 & =& \begin{array}{|r|r|r|r|r|r|}\hline
2 & -3 & 6 & -7 & -10 & 12\\\hline 
-4 & 5 & -8 & 9 & 11 & -13 \\\hline 
\end{array}\;, & 
V_3 & =& \begin{array}{|r|r|r|r|r|r|}\hline
4 & 7 & -13 & 12 & -8 & -2 \\ \hline
-6 & -5 & 11 & -10 & 9 & 1 \\ \hline  
\end{array}\;,\\[8pt]
V_5 & =& \begin{array}{|r|r|r|r|r|r|}\hline
1 & -2 & 6 & -7 & -10 & 12\\\hline 
-3 & 4 & -8 & 9 & 11 & -13 \\\hline 
\end{array}\;, &
V_7 & =& \begin{array}{|r|r|r|r|r|r|}\hline
1 & -2 & 8 & 11 & -5 & -13 \\ \hline
-3 & 4 & -10 & -9 & 6 & 12 \\\hline
\end{array}\;,\\[8pt]
V_9 & =& \begin{array}{|r|r|r|r|r|r|}\hline
1 & -2 & 5 & -6 & -10 & 12\\\hline 
-3 & 4 & -7 & 8 & 11 & -13 \\\hline 
\end{array}\;,& 
V_{11} & =& \begin{array}{|r|r|r|r|r|r|}\hline
1 & -2 & -10 & -7 & 13 & 5\\ \hline
-3 & 4 & 8 & 9 & -12 & -6 \\ \hline
\end{array}\;,\\[8pt]
V_{13} & =& \begin{array}{|r|r|r|r|r|r|}\hline
1 & -2 & 5 & -6 & -9 & 11 \\\hline 
-3 & 4 & -7 & 8 & 10 & -12\\\hline 
\end{array}\;.
\end{array}$$
\end{footnotesize}
\caption{Shiftable blocks of size $2\times 6$ satisfying condition \eqref{blocchi}.}\label{bl6}
\end{figure}

In several cases, we provide a basic sequence $S$  of blocks having disjoint supports.
To simplify the notation, if $S=(B_1,B_2,\ldots,B_r)$  we write $\supp(S)$ instead of $\bigcup\limits_{i=1}^r 
\supp(B_i)$.
Also,  recall that $S\pm x$ means the sequence $(B_1\pm x, \ldots,B_r\pm x)$. 
It will be very useful to define the following sequence. For all $b\geq 1$ let
$$U(b)=(V_{13},V_{13}\pm 12, V_{13}\pm 24, \ldots, V_{13}\pm 12(b-1)).$$
Also, we set $U(0)$ to be the empty sequence: so, for all $b\geq 0$ the sequence $U(b)$ contains $b$ elements
and $\supp(U(b))=[1,12b]$.

If $\rho=3$ we take $\F=\con\limits_{c=0}^{h-1} (F_3\pm 18c)$;
if $\rho=5$ we take $\F=\con\limits_{c=0}^{h-1} (F_5\pm 15c)$. 

If $\rho=12x+7$, we take $S=(U(x), V_7\pm 12x, U(x)\pm (12x+13))$, 
whence $\supp(S)=[1,2\rho]\setminus\{\rho,2\rho \}$.
In this case we define $\F$ as the sequence of the first $h$ elements of 
$\con\limits_{c= 0}^{\Sh} (S\pm 2\rho c)$. 

If $\rho=12x+9$, we take 
$$S=(U(x), V_9\pm 12x, U(x) \pm(12x+13),   V_5\pm (24x+13), U(x)\pm (24x+26)).$$
In fact, $\supp(S)=[1,3\rho ]\setminus \{\rho, 2\rho, 3\rho \}$. 
In this case we define $\F$ as the sequence of the first $h$ elements of 
$\con\limits_{c= 0}^{\Sh} (S\pm 3\rho c)$. 

If $\rho=12x+11$, we take
$$\begin{array}{rcl}
S & =& ( U(x), V_{11}\pm 12x, U(x)\pm (12x+13), V_9\pm (24x+13), U(x)\pm (24x+26), \\
&& V_7\pm (36x+26), U(x)\pm (36x+39), V_5\pm (48x+39), U(x)\pm (48x+52),  \\
&& V_3\pm (60x+52), U(x)\pm (60x+65)). 
  \end{array}
$$
We obtain $\supp(S)=[1,6\rho]\setminus \{j\rho: j \in [1,6]  \}$, and then
$\F$ is the sequence consisting of the first $h$ elements of $\con\limits_{c= 0}^{\Sh} (S\pm 6\rho c)$. 

If $\rho=12x+13$, we take $S=(U(x),V_{13}\pm 12x )$, 
whence $\supp(S)=[1,\rho]\setminus\{\rho \}$.
In this case we define $\F$ as the sequence of the first $h$ elements of 
$\con\limits_{c=0}^{\Sh} (S\pm \rho c)$. 

If $\rho=12x+15$, we take
$$\begin{array}{rcl}
S & =& ( U(x+1), V_3 \pm 12(x+1), U(x)\pm (12x+25), V_5\pm (24x+25), 
U(x)\pm (24x+38), \\
&& V_7 \pm (36x+38), U(x) \pm (36x+51), 
V_9\pm (48x+51), U(x)\pm (48x+64),\\
&& V_{11}\pm (60x+64), U(x+1)\pm (60x+77)). 
  \end{array}
$$
We get $\supp(S)=[1, 6\rho ]\setminus \{j\rho: j \in [1,6]  \}$. 
In this case we define $\F$ as the sequence of the first $h$ elements of 
$\con\limits_{c=0}^{\Sh} (S\pm 6\rho c)$. 

Finally, if $\rho=12x+17$, we take
$$S=(U(x+1), V_5 \pm 12(x+1), U(x)\pm (12x+25), V_9\pm (24x+25), U(x+1)\pm (24x+38)).$$
In fact, $\supp(S)=[1,3\rho ]\setminus \{\rho, 2\rho, 3\rho \}$. 
In this case we define $\F$ as the sequence of the first $h$ elements of 
$\con\limits_{c=0}^{\Sh} (S\pm 3\rho c)$. 
\end{proof}

\begin{ex}
Following the proof of  Lemma \ref{F_rho}, we obtain the following sequence
$\F_{5,15}=(B_1,\ldots,B_{5} )$ of shiftable blocks:

\begin{footnotesize}
$$\begin{array}{rcl}
B_1 & =& \begin{array}{|c|c|c|c|c|c|}\hline
 1& -2& 5& -6& -9& 11 \\\hline
-3& 4& -7& 8& 10& -12 \\\hline 
\end{array}\;, \\[8pt]
B_2 & =& \begin{array}{|c|c|c|c|c|c|}\hline
 16& 19& -25& 24& -20& -14 \\\hline
 -18& -17& 23& -22& 21& 13 \\\hline 
\end{array}\;, \\[8pt]
B_3 & =& \begin{array}{|c|c|c|c|c|c|}\hline
 26& -27& 31& -32& -35& 37\\\hline
   -28& 29& -33& 34& 36& -38 \\\hline
\end{array}\;, \\[8pt]
B_4 & =& \begin{array}{|c|c|c|c|c|c|}\hline
 39& -40& 46& 49& -43& -51 \\\hline
   -41& 42& -48& -47& 44& 50 \\\hline
\end{array}\;, \\[8pt]
B_5 & =& \begin{array}{|c|c|c|c|c|c|}\hline
 52& -53& 56& -57& -61& 63 \\\hline
 -54& 55& -58& 59& 62& -64 \\\hline
\end{array}\;.
\end{array}$$
\end{footnotesize}
\end{ex}

\begin{prop}\label{G_rho-6}
Let $m$ and $s$ be even integers with $m\geq 2$, $s\geq 6$ and $s\equiv 0 \pmod 6$.
Let $t$ be a divisor of $ms$ and set $\ell=\frac{2ms}{t}+1$.
There exists a sequence $\B$ of $\frac{m}{2}$ blocks of size $2\times s$
such that $\B$ satisfies condition \eqref{blocchi} and
$\supp(\B)=\left[ 1,ms+\left\lfloor t/2
\right\rfloor \right]
\setminus \left\{j\ell: j \in \left[1, \left\lfloor t/2 \right\rfloor \right]\right\}$.
\end{prop}

\begin{proof}
Write $s=6q$, with $q\geq 1$: by Lemma \ref{F_rho}, we can construct a sequence 
$\F=\F_{\frac{mq}{2}, \ell}=(A_1,\ldots,A_{\frac{mq}{2}})$ 
of shiftable blocks of size $2\times 6$ in such a way that $\F$ satisfies condition 
\eqref{blocchi}  and 
$\supp(\F)=\left[ 1,ms+\left\lfloor t/2
\right\rfloor \right]
\setminus \left\{j\ell: j \in \left[1, \left\lfloor t/2 \right\rfloor \right]\right\}$.
So, for all $i \in [1,\frac{m}{2}]$, let $B_i$ be the block of size $2\times s$ obtained by juxtaposing the $q$ blocks
of $\F$
$$A_{(i-1)q+1},\;A_{(i-1)q+2},\;A_{(i-1)q+3},\; \ldots,A_{iq}.$$
By construction, the sequence $\B=(B_1,\ldots,B_{\frac{m}{2}})$ satisfies condition \eqref{blocchi}
and has cardinality $\frac{m}{2}$.
Furthermore, since we used all the blocks of $\F$, we obtain that  $\supp(\B)=\supp(\F)$.
\end{proof}

\begin{prop}\label{G_rho-8}
Let $m$ and $s$ be even integers with $m\geq 2$, $s\geq 8$  and $s\equiv 2 \pmod 6$. Let $t$ be a divisor of $ms$
and set $\ell=\frac{2ms}{t}+1$.
There exists a sequence $\B$ of $\frac{m}{2}$ blocks of size $2\times s$ such  that $\B$ satisfies 
condition \eqref{blocchi} and
$\supp(\B)=[1,ms+\lfloor t/2\rfloor]
\setminus \{j\ell: j \in [1,\lfloor t/2 \rfloor ]\}$.
\end{prop}

\begin{proof}
Write $s=6q+8$ and  $N=6mq+\eta$, where $\eta=\left\lfloor \frac{6mq}{\ell-1}\right\rfloor$.
By Lemma \ref{F_rho}, we can construct a sequence $\F=\F_{\frac{mq}{2},\ell}$ satisfying condition \eqref{blocchi} and 
whose elements are shiftable blocks of size $2\times 6$.
We also have $\supp(\F)=[1,N]\setminus\{j\ell: j \in [1,\eta] \}$.
Now, we have to construct a sequence $\G$ of shiftable blocks of size $2\times 8$ satisfying condition \eqref{blocchi}
in such a way that $|\G|=\frac{m}{2}$ and 
$$\supp(\G)=[N+1,ms+\lfloor t/2\rfloor]\setminus \{j\ell: j \in [\eta+1,\lfloor t/2 \rfloor ]\}.$$
To this purpose, we use the shiftable blocks of size $2\times 8$ of Figure \ref{bl8}.
Each of these blocks has rows that sum to zero and columns 
with the following sums:
$$F_3:  (-3, 3, 3, -3, -3, 3, 3, -3),\qquad 
F_5, W_{i,j}, V_i:  (-2,2,1,-1,-2,2,1,-1).$$
Furthermore, observe that
$$\begin{array}{rcl}
\supp (F_3)&=& [1,24]\setminus \{3,6,9,12,15,18,21,24\},\\
\supp (F_5)&=&[1,20]\setminus \{5,10,15,20\},\\
\supp (W_{7,3})&=&[1,19] \setminus \{3,10,17\},\\
\supp (W_{7,i})&=&[1,18] \setminus \{i,7+i\}, \quad  i=5,7, \\
 \supp (W_{9,5})&=&[1,18] \setminus \{5,14\},\\
 \supp (W_{11,i})&=&[1,18] \setminus \{i,11+i\}, \quad  i=3,5, \\
 \supp (V_{j})&=&[1,17] \setminus \{j\}, \quad j=1,3,5,7,9,11,13,15,17.
\end{array}$$
Keeping the strategy of Lemma \ref{F_rho}, we describe basic sequences $S$ of blocks.

If $\ell=3$, then $t=ms$ and $N=9mq\equiv 0 \pmod 3$. 
We can take
$\G=\con\limits_{c=0}^{m/2-1}(F_3\pm (N+24c))$.
In fact,  we have $\supp(\G)=\left[9mq+1, \frac{3ms}{2}\right]\setminus \left\{3j : 
j \in  \left[3mq+1, \frac{ms}{2}\right] \right\}$, as required.
If $\ell=5$, then $N=15\frac{mq}{2}\equiv 0 \pmod 5$ and we can take 
$\G=\con\limits_{c=0}^{m/2-1}(F_5\pm (N +20c ))$.
We obtain $\supp(\G)=\left[15\frac{mq}{2}+1, 5\frac{ms}{4}\right]\setminus \left\{5j : 
j \in  \left[\frac{3mq}{2}+1, \frac{ms}{4}\right] \right\}$.

If $\ell=7$, then $3$ divides $m$ and $N=7mq\equiv 0 \pmod 7$.
We can take $S=(W_{7,7}, W_{7,3}\pm 18,W_{7,5}\pm 37 )$, as
$\supp(S)=[1, 8\ell]\setminus\{j \ell : j \in [1,8] \}$.
Since $|S|=3$, we take $\G=\con\limits_{c=0}^{m/6-1} (S\pm (N+8\ell c))$.

If $\ell=13$, then $3$ divides $m$ and $N=13\frac{mq}{2}\equiv 0 \pmod{13}$.
Let $S= (V_{13}, V_{9}\pm 17, V_{5}\pm 34)$: 
we have $\supp(S)=[1,  4\ell]\setminus \{j\ell: j \in [1,4] \}$.
Hence $\G=\con\limits_{c=0}^{m/6-1} (S\pm (N+4\ell c))$.

If $\ell=9$, we have to distinguish two cases.
\begin{itemize}
 \item[(1)] If $mq\equiv 0 \pmod 4$, then $N=27\frac{mq}{4}\equiv 0 \pmod 9$.
Hence, we can take $\G=\con\limits_{c=0}^{m/2-1}( V_9\pm (N+2\ell c) )$. Note that 
$\supp(V_9)=[1,2\ell]\setminus\{\ell,2\ell\}$.
\item[(2)] If $mq\equiv 2 \pmod 4$, then $N\equiv 4 \pmod 9$.
In this case, we take $\G=\con\limits_{c=0}^{m/2-1} (W_{9,5}\pm (N+2\ell c))$:
recall that $\supp(W_{9,5})=[1,2\ell]\setminus\{\ell-4, 2\ell-4 \}$.
\end{itemize}

For  $\ell=11$  we have five possibilities.
\begin{itemize}
 \item[(1)] If $mq\equiv 0 \pmod{10}$ then $N\equiv 0 \pmod \ell$, so we can take
$S=(V_{11}, W_{11,5}\pm 17, V_{9}\pm 35, W_{11,3}\pm 52, V_{7}\pm 70 )$.
We have  $\supp(S)=[1, 8\ell]\setminus \{j\ell: j\in [1,8]  \}$.
 \item[(2)] If  $mq\equiv 2\pmod{10}$, then $N\equiv 2 \pmod \ell$, so we can take
$S=( V_{9}, W_{11,3}\pm 17, V_{7}\pm 35, V_{11}\pm 53, W_{11,5}\pm 70 )$.
We get  $\supp(S)=[1,  8\ell]\setminus \{j\ell-2 : j\in [1,8] \}$.
\item[(3)]  If  $mq\equiv 4\pmod{10}$, then $N\equiv 4 \pmod \ell$, so we can take
$S=( V_{7}, V_{11} \pm 18, W_{11,5}\pm 35, V_{9}\pm 53, W_{11,3}\pm 70 )$.
Note that  $\supp(S)=[1, 8\ell ]\setminus \{j\ell-4: j\in [1,8]\}$.
\item[(4)]  If  $mq\equiv 6\pmod{10}$, then $N\equiv 6 \pmod \ell$, so we can take
$S=( W_{11,5}, V_{9} \pm 18, W_{11,3}\pm 35, V_{7}\pm 53, V_{11}\pm 71 )$.
We have  $\supp(S)=[1, 8\ell]\setminus \{ j\ell-6: j\in [1,8]\}$.
\item[(5)]  If  $mq\equiv 8\pmod{10}$, then $N\equiv 8 \pmod \ell$, so we can take
$S=(W_{11,3}, V_{7} \pm 18, V_{11}\pm 36, W_{11,5}\pm 53, V_{9}\pm 71  )$.
We get  $\supp(S)=[1,  8\ell]\setminus \{j\ell-8: j\in [1,8]\}$.
 \end{itemize}
In all five cases, $\G$ consists of the first $\frac{m}{2}$ elements of 
$\con\limits_{c=0}^{\lfloor m/10 \rfloor} (S\pm (N+8\ell c))$. 

For  $\ell=15$  we have seven possibilities.
\begin{itemize}
 \item[(1)] If $mq\equiv 0 \pmod{14}$ then $N\equiv  0\pmod \ell$, so we can take
$S=( V_{15}, V_{13}\pm 17, V_{11}\pm 34, V_{9}\pm 51,
V_{7}\pm 68,  V_{5}\pm 85, V_{3}\pm 102 )$.
In fact,  $\supp(S)=[1,    8\ell]\setminus \{ j\ell: j\in [1,8]  \}$.
 \item[(2)] If $mq\equiv 2 \pmod{14}$ then $N\equiv  12\pmod \ell$, so we can take
$S=( V_{3}, V_{15}\pm 18, V_{13}\pm 35, V_{11}\pm 52,
V_{9}\pm 69, V_{7}\pm 86, V_{5}\pm 103 )$.
We get   $\supp(S)=[1, 8\ell ]\setminus \{ j\ell-12:j\in [1,8] \}$.
 \item[(3)] If $mq\equiv 4 \pmod{14}$ then $N\equiv  10 \pmod \ell$, so we can take
$S=( V_{5}, V_{3}\pm 17, V_{15}\pm 35, V_{13}\pm 52,
V_{11}\pm 69, V_{9}\pm 86, V_{7}\pm 103 )$.
In fact,  $\supp(S)=[1,8\ell ]\setminus \{j\ell-10: j\in [1,8]  \}$.
 \item[(4)] If $mq\equiv 6 \pmod{14}$ then $N\equiv  8 \pmod \ell$, so we can take
$S=( V_{7}, V_{5}\pm 17, V_{3}\pm 34 ,V_{15}\pm 52,
V_{13}\pm 69, V_{11}\pm 86, V_{9}\pm 103 )$.
We have  $\supp(S)=[ 1, 8\ell]\setminus \{ j\ell-8 : j\in [1,8]\}$.
 \item[(5)] If $mq\equiv 8 \pmod{14}$ then $N\equiv 6  \pmod \ell$, so we can take
$S=( V_{9}, V_{7}\pm 17, V_{5}\pm 34, V_{3}\pm 51,
V_{15}\pm 69, V_{13}\pm 86, V_{11}\pm 103)$.
In fact,  $\supp(S)=[ 1, 8\ell]\setminus \{j\ell-6 :j\in [1,8]\}$.
 \item[(6)] If $mq\equiv 10 \pmod{14}$ then $N\equiv 4  \pmod \ell$, so we can take
$S=(V_{11}, V_{9}\pm 17, V_{7}\pm 34, V_{5}\pm 51,
V_{3}\pm 68, V_{15}\pm 86, V_{13}\pm 103 )$.
In fact,  $\supp(S)=[ 1, 8\ell]\setminus \{j\ell-4 :j\in [1,8]\}$.
 \item[(7)] If $mq\equiv 12 \pmod{14}$ then $N\equiv 2  \pmod \ell$, so we can take
$S=(V_{13}, V_{11}\pm 17, V_{9}\pm 34, V_{7}\pm 51,
V_{5}\pm  68, V_{3}\pm 85, V_{15}\pm 103 )$.
In fact,  $\supp(S)=[ 1,  8\ell]\setminus \{j\ell-2 :j\in [1,8] \}$.
\end{itemize}
In all seven cases, $\G$ consists of the first $\frac{m}{2}$ elements of 
$\con\limits_{c=0}^{\lfloor m/14 \rfloor} (S\pm (N+8\ell c))$.

Suppose now that $\ell\geq 17$: in this case, any set of $16$ consecutive integers contains at
most one multiple of $\ell$. We start considering the interval $[N+1,N+16]$ and the first multiple of $\ell$ belonging to the interval $[N+1,ms+\lfloor t/2\rfloor]$. So, if $(\eta+1)\ell$ is an element of $[N+1,N+16]$ we take the block $V_{r}$ where $r$ must be chosen in such a way that 
$\supp(V_r)$ does not contain $(\eta+1)\ell$. Otherwise, we take the block $V_{13}$ and repeat this process considering the interval $[N+17,N+32]$.

It will be useful to  define,  for all $b\geq 1$,  the sequence
$$U(b)=(V_{17},V_{17}\pm 16, V_{17} \pm 32,\ldots, V_{17}\pm 16(b-1)).$$
Also, we set $U(0)$ to be the empty sequence: so, for all $b\geq 0$ the sequence $U(b)$ contains $b$ elements and $\supp(U(b))=[1,16b]$.

Write $(\eta+1)\ell-N=16h_0+r_0$, where $0\leq r_0 < 16$, and define the sequence
$$S_0=(U(h_0), V_{r_0}\pm 16h_0 ).$$
Note that $r_0$ is odd, since $\ell$ is odd and $(\eta+1)\ell-N \equiv (\eta+1)\ell+\eta \equiv 1\pmod 2$.
Furthermore, $\supp(S_0\pm N)=[N+1, N+16h_0+17]\setminus \{ (\eta+1)\ell \}$.

Now, for all $j\in [1,\lfloor t/2 \rfloor - \eta]$, write 
$\ell-17+r_{j-1}=16h_j +r_j$, where $0\leq  r_j < 16$, and define the sequence
$$S_j=\left(U(h_j)\pm\left(17j+16\sum_{i=0}^{j-1} h_i   \right),
V_{r_j}\pm  \left(17j+16\sum_{i=0}^j h_i  \right)\right).$$
Note that $(\eta+j+1)\ell-N=16\sum_{i=0}^{j} h_i +17j+ r_j $ and
$$\supp(S_j\pm N)=\left[N+1+17j+16\sum_{i=0}^{j-1} h_i,\;  N+17(j+1)+ 16\sum_{i=0}^j h_i  \right]
\setminus \{(\eta+j+1)\ell\}.$$
The elements of $\G$ are the first $\frac{m}{2}$ blocks in 
$\con\limits_{c=0}^{\lfloor t/2 \rfloor - \eta} (S_c\pm N)$.

Finally, writing $\F=(A_1,\ldots,A_{\frac{mq}{2}})$ and $\G=(G_1,\ldots,G_{\frac{m}{2}} )$,
for all $i=1,\ldots,\frac{m}{2}$, let $B_i$ be the block of size $2\times s$ obtained by juxtaposing the $q$ blocks 
$$A_{(i-1)q+1},\;A_{(i-1)q+2},\;A_{(i-1)q+3},\; \ldots,A_{iq}$$
and the  block $G_i$.
By construction, the sequence $\B=(B_1,\ldots,B_{\frac{m}{2}})$ satisfies condition \eqref{blocchi}, 
has cardinality $\frac{m}{2}$ and  
$\supp(\B)=\supp(\F)\cup \supp(\G)=[ 1,ms+\lfloor t/2 \rfloor ]
\setminus \{j\ell: j \in [1,\lfloor t/2\rfloor]\}$.
\end{proof}

\begin{figure}[ht]
\begin{footnotesize}
$$\begin{array}{rcl}
F_{3} & =& \begin{array}{|c|c|c|c|c|c|c|c|}\hline
1 & -2 & -7 & 8    & 13 & -14 &  -19 & 20 \\ \hline
-4 & 5 &  10 & -11 & -16 & 17 & 22 & -23 \\\hline
\end{array}\;,\\[8pt]
F_5 & = & \begin{array}{|c|c|c|c|c|c|c|c|}\hline
1 & -2 &  19 & -14 & 7 & -6 & 12 & -17 \\ \hline
-3 & 4 & -18 & 13 & -9 & 8 & -11 & 16 \\ \hline
      \end{array}\;,\\[8pt]
      W_{7,3} & =& \begin{array}{|c|c|c|c|c|c|c|c|}\hline
4 & 7 & -8 & -12 & 13 & 16 & -1 & -19 \\ \hline
-6 & -5 & 9 & 11 & -15 & -14 & 2 & 18 \\ \hline
         \end{array}\;,\\[8pt]
W_{7,5} & = & \begin{array}{|c|c|c|c|c|c|c|c|}\hline
   1 & -2 & -6 & 8 & -15 & -14 & 18 & 10  \\\hline
   -3 & 4 & 7 & -9 & 13 & 16 & -17 & -11  \\\hline
         \end{array}\;,\\[8pt]
W_{7,7}  & = & \begin{array}{|c|c|c|c|c|c|c|c|}\hline
  1 & -2 &  6 & -16 & -11 & -8 & 13 & 17     \\\hline
  -3 & 4 & -5 & 15 & 9 & 10 & -12 & -18     \\\hline
         \end{array}\;,\\[8pt]
 W_{9,5} & =& \begin{array}{|c|c|c|c|c|c|c|c|}\hline
 1 & -2 & -6 & 8 & -12 & 13 & -17 & 15 \\\hline
 -3 & 4 & 7 & -9 & 10 & -11 & 18 & -16 \\\hline
               \end{array}\;,\\[8pt]
W_{11,3}  & =&  \begin{array}{|c|c|c|c|c|c|c|c|}\hline
5 & -8 & -1 & -13 & 9 & 6 & 18 & -16 \\\hline
-7 & 10 & 2 & 12 & -11 & -4 & -17 & 15 \\\hline
                \end{array}\;,\\[8pt]
W_{11,5} & =&  \begin{array}{|c|c|c|c|c|c|c|c|}\hline
  1 & -2 & -6 & 8 & 13 & 14 & -10 & -18 \\\hline
  -3 & 4 & 7 & -9 & -15 & -12 & 11 & 17 \\\hline
                \end{array}\;,\\[8pt]
V_1 & = & \begin{array}{|c|c|c|c|c|c|c|c|}\hline
2 & -3 &  17 & -13 & 7 & -6 & 11 & -15 \\ \hline
-4 & 5 & -16 & 12 & -9 & 8 & -10 & 14 \\ \hline
      \end{array}\;,\\[8pt]
  V_3 & = & \begin{array}{|c|c|c|c|c|c|c|c|}\hline
4 & 7 & -1 & 14 & -12 & 13 & -8 & -17 \\\hline
-6 & -5 & 2 & -15 & 10 & -11 & 9 & 16 \\ \hline
         \end{array}\;,\\[8pt]
         V_5 & = & \begin{array}{|c|c|c|c|c|c|c|c|}\hline
1 & -2 &  17 & -13 & 7 & -6 & 11 & -15 \\ \hline
-3 & 4 & -16 & 12 & -9 & 8 & -10 & 14 \\ \hline
         \end{array}\;,\\[8pt]
V_7 & = & \begin{array}{|c|c|c|c|c|c|c|c|}\hline
1 & -8 &  6 & 16 & -11 & -2 & 13 & -15 \\\hline
-3 & 10 & -5 & -17 & 9 & 4 & -12 & 14 \\\hline 
         \end{array}\;,\\[8pt]
V_9 & = & \begin{array}{|c|c|c|c|c|c|c|c|}\hline
1 & -2 & 17 & -13 & 6 & -5 & 11 & -15 \\ \hline
-3 & 4 & -16 & 12 & -8 & 7 & -10 & 14 \\ \hline
         \end{array}\;,\\[8pt]
V_{11} & = & \begin{array}{|c|c|c|c|c|c|c|c|}\hline
1 & -2 & 17 & -6 & -15 & -12 & 8 & 9 \\\hline
-3 & 4 & -16 & 5 & 13 & 14 & -7 & -10 \\ \hline
         \end{array}\;,\\[8pt]
V_{13} & = & \begin{array}{|c|c|c|c|c|c|c|c|}\hline
1 & -2 & 17 & 9 & -7 & 8 & -11 & -15 \\\hline
-3  & 4 & -16 & -10 & 5 & -6 & 12 & 14 \\\hline 
         \end{array}\;,\\[8pt]
V_{15} & = & \begin{array}{|c|c|c|c|c|c|c|c|}\hline
1 & -2 & -5 & 7 & 11 & 14 & -16 & -10 \\ \hline
-3 & 4 & 6 & -8 & -13 & -12 & 17 & 9  \\\hline
         \end{array}\;,\\[8pt]
V_{17} & = & \begin{array}{|c|c|c|c|c|c|c|c|}\hline
1 & -2 &  16 & -12 & 6 & -5 & 10 & -14 \\ \hline
-3 & 4 & -15 & 11 & -8 & 7 & -9 & 13 \\ \hline
        \end{array}\;.
\end{array}$$
\end{footnotesize}
\caption{Shiftable blocks of size $2\times 8$ satisfying condition \eqref{blocchi}.}\label{bl8}
\end{figure}

\begin{prop}\label{G_rho-10}
Let $m$ and $s$ be even integers with $m\geq 2$, $s\geq 10$ and $s\equiv 4 \pmod 6$.
Let $t$ be a divisor of $ms$ and set $\ell=\frac{2ms}{t}+1$.
There exists a sequence $\B$ of $\frac{m}{2}$ blocks of size $2\times s$ such  that $\B$ satisfies 
condition \eqref{blocchi} and
$\supp(\B)=[ 1,ms+\lfloor t/2\rfloor] \setminus \{j\ell: j \in [1, \lfloor t/2 \rfloor]\}$.
\end{prop}

\begin{proof}
This proof is very similar to that of Proposition \ref{G_rho-8},
so we can skip some details. Write $s=6q+10$ and  $N=6mq+\eta$, where $\eta=\left\lfloor \frac{6mq}{\ell-1}\right\rfloor$.
Let $\F=\F_{\frac{mq}{2},\ell}$ be the sequence satisfying condition \eqref{blocchi} constructed 
in Lemma \ref{F_rho}. 
We have to construct a sequence $\G$ of shiftable blocks of size $2\times 10$ satisfying condition \eqref{blocchi}
in such a way that $|\G|=\frac{m}{2}$ and 
$$\supp(\G)=[N+1,ms+\lfloor t/2\rfloor]\setminus \{j\ell: j \in [\eta+1,\lfloor t/2 \rfloor ]\}.$$
To this purpose, we use the shiftable blocks of size $2\times 10$ of Figure \ref{bl10}.
Each of these blocks has rows that sum to zero and columns 
with the following sums:
$$\begin{array}{c}
F_3:(-1,1, -3, 3, 6, -6, -3, 3, 3, -3 ),
\qquad F_5: (-2, 2, -2, 2, 1, -1, -2, 2, 2, -2), \\
W_{7,i}: (1, -1, -2, 2, 1, -1, -2, 2, 2, -2), \\
W_{9,i}, W_{13,5}, W_{15,i},  V_j: (-2, 2, -2, 2, 1, -1, -2, 2, 2, -2).
\end{array}$$
Observe that
$$\begin{array}{rcl}
\supp (F_3)&=& [1,30]\setminus \{3,6,9,12,15,18,21,24,27,30\},\\
\supp (F_5)&=&[1,25]\setminus \{5,10,15,20,25\},\\
\supp (W_{7,i})&=&[1,23] \setminus \{i,7+i,14+i\}, \quad  i=3,5,7, \\
 \supp (W_{9,i})&=&[1,22] \setminus \{i,9+i\}, \quad i=5,9,\\
 \supp (W_{13,5})&=&[1,22] \setminus \{5,18\},  \\
 \supp (W_{15,i})&=&[1,22] \setminus \{i,15+i\}, \quad i=3,5,\\
 \supp (V_{j})&=&[1,21] \setminus \{j\}, \quad j=1,3,5,7,9,11,13,15,17,19,21.
\end{array}$$
Furthermore, if $\ell\in \{3,5,7,11, 13,19\}$, then $N\equiv 0 \pmod \ell$.

If $\ell=3$, let $S=(F_3)$ and  $\G=\con\limits_{c=0}^{m/2-1}
 (S\pm (N+24 c))$. Similarly,  if $\ell=5,11$, let $S=(F_5)$ and $S=(V_{11})$ respectively.

If $\ell=7$, let  $S= (W_{7,7}, W_{7,5}\pm 23, W_{7,3}\pm 46  )$.
Note that $\supp(S)=[1, 10\ell]\setminus \{j\ell: j\in [1,10]\}$.
In this case, $\G=\con\limits_{c=0}^{m/6-1} (S\pm (N+10\ell c))$.

If $\ell=13$, let $S=(V_{13}, W_{13,5}\pm 21, V_{9}\pm 43)$.
Note that $\supp(S)=[1, 5\ell]\setminus \{j\ell: j\in [1,5]\}$.
We can take $\G=\con\limits_{c=0}^{m/6-1} (S\pm (N+5\ell c))$. 

If $\ell=19$ let 
$$S=(V_{19}, V_{17}\pm 21, V_{15}\pm 42,
V_{13}\pm 63,V_{11}\pm 84, V_{9}\pm 105, V_{7}\pm 126,
V_{5}\pm 147, V_{3}\pm 168).$$
We have $\supp(S)=[1, 10\ell]\setminus \{j\ell : j\in [1,10]\}$
and $\G=\con\limits_{c=0}^{m/18-1} (S\pm (N+10\ell c))$.

For  $\ell=9$  we have two possibilities.
\begin{itemize}
\item[(1)] If $mq \equiv 0 \pmod 4$, then $N\equiv 0 \pmod \ell$ and we can take
$S= (W_{9,9}, W_{9,5}\pm 22)$. In fact,  $\supp(S)=[1,5\ell]\setminus \{j\ell : j\in [1,5] \}$.
\item[(2)] If $mq \equiv 2 \pmod 4$, then $N\equiv 4\pmod \ell$.
We can take $S=(W_{9,5} , W_{9,9}\pm 23)$, in fact
$\supp(S)=[1,5\ell]\setminus \{j\ell-4 : j\in [1,5]\}$.
\end{itemize}
In both cases, the sequence $\G$ consists of the first $\frac{m}{2}$ elements of $\con\limits_{c=0}^{\lfloor m/4 \rfloor} (S\pm (N+5\ell c))$. 

If $\ell=15$ we have seven possibilities.
\begin{itemize}
 \item[(1)] If $mq\equiv 0 \pmod{14}$ then $N\equiv 0 \pmod \ell$ and we can take
 $S=(V_{15}, V_9\pm 21,  W_{15,3}\pm 42,
 V_{11}\pm 64, W_{15,5}\pm 85, V_{13}\pm 107, V_{7}\pm 128)$.
 In fact,  $\supp(S)=[1,10\ell]\setminus \{j\ell : j\in [1,10] \}$.
\item[(2)] If $mq\equiv 2\pmod{14}$ then $N\equiv 12 \pmod \ell$ and we can take
$S=(W_{15,3}, V_{11}\pm 22,  W_{15,5}\pm 43, V_{13}\pm 65,
V_{7}\pm 86,V_{15}\pm 108, V_{9}\pm 129)$.
In fact,  $\supp(S)=[1,10\ell]\setminus \{j\ell-12: j\in [1,10] \}$.
\item[(3)] If $mq\equiv 4\pmod{14}$ then $N\equiv 10 \pmod \ell$ and we can take
$S=( W_{15,5}, V_{13}\pm 22,  V_7\pm 43, V_{15}\pm 65, V_{9}\pm 86,
W_{15,3}\pm 107, V_{11}\pm 129)$.
In fact,  $\supp(S)=[1,10\ell]\setminus \{j\ell-10 : j\in [1,10] \}$.
\item[(4)] If $mq\equiv 6\pmod{14}$ then $N\equiv 8 \pmod \ell$ and we can take
$S=( V_{7}, V_{15}\pm 22,  V_9\pm 43, W_{15,3}\pm 64, V_{11}\pm 86,
W_{15,5}\pm 107, V_{13}\pm 129)$.
In fact,  $\supp(S)=[1, 10\ell]\setminus \{j\ell-8 : j\in [1,10] \}$.
 \item[(5)] If $mq\equiv 8\pmod{14}$ then $N\equiv 6 \pmod \ell$ and we can take
$S=( V_{9}, W_{15,3}\pm 21,  V_{11}\pm 43, W_{15,5}\pm 64, V_{13}\pm 86,
V_{7}\pm 107, V_{15}\pm 129)$.
In fact,  $\supp(S)=[1,10\ell]\setminus \{j\ell-6 : j\in [1,10] \}$.
 \item[(6)] If $mq\equiv 10\pmod{14}$ then $N\equiv 4 \pmod \ell$ and we can take
$S=(V_{11}, W_{15,5}\pm 21,  V_{13}\pm 43, V_{7}\pm 64, V_{15}\pm 86,
V_{9}\pm 107, W_{15,3}\pm 128)$.
In fact,  $\supp(S)=[1, 10\ell]\setminus \{j\ell-4 : j\in [1,10] \}$.
\item[(7)] If $mq\equiv 12\pmod{14}$ then $N\equiv 2 \pmod \ell$ and we can take
$S=( V_{13}, V_{7}\pm 21,  V_{15}\pm 43, V_{9}\pm 64, W_{15,3}\pm 85,
V_{11}\pm 107, W_{15,5}\pm 128)$.
In fact,  $\supp(S)=[1, 10\ell]\setminus \{j\ell-2 : j\in [1,10] \}$.
\end{itemize}
In all seven cases, the sequence $\G$ consists of the first $\frac{m}{2}$ elements of $\con\limits_{c=0}^{\lfloor m/14 \rfloor} (S\pm (N+10\ell c))$.

If $\ell=17$ we have four possibilities.
\begin{itemize}
\item[(1)] If $mq\equiv 0 \pmod 8$ then $N\equiv 0 \pmod \ell$ and we can take
$S=(V_{17}, V_{13}\pm 21,  V_{9}\pm 42, V_{5}\pm 63)$.
We have $\supp(S)=[1, 5\ell]\setminus \{j \ell: j \in [1,5] \}$.
\item[(2)] If $mq \equiv 2 \pmod 8$ then $N\equiv 12 \pmod \ell$ and we can take
$S=( V_{5}, V_{17}\pm 22,  V_{13}\pm 43,  V_{9}\pm 64)$.
We get $\supp(S)=[1,5\ell]\setminus \{j\ell-12:  j \in [1,5]\}$.
\item[(3)] If $mq \equiv 4 \pmod 8$ then $N\equiv 8 \pmod \ell$ and we can take
$S=(V_{9}, V_{5}\pm 21,  V_{17}\pm 43,  V_{13}\pm 64)$.
We have $\supp(S)=[1, 5\ell]\setminus \{j\ell-8 : j \in [1,5]\}$.
\item[(4)] If $mq \equiv 6 \pmod 8$ then $N\equiv 4 \pmod \ell$ and we can take
$S=(V_{13}, V_{9}\pm 21,  V_{5}\pm 42,  V_{17}\pm 64)$.
We obtain $\supp(S)= [1, 5\ell]\setminus \{j\ell-4 : j \in [1,5]\}$.
\end{itemize}
In all four cases, the sequence $\G$ consists of the first $\frac{m}{2}$ 
elements of $\con\limits_{c=0}^{\lfloor m/8 \rfloor} (S\pm (N+5\ell c))$.

Suppose now that $\ell\geq 21$: in this case, any set of $20$ consecutive integers contains at
most one multiple of $\ell$. For all $b\geq 1$ we define the sequence
$$U(b)=(V_{21},V_{21}\pm 20, V_{21} \pm 40,\ldots, V_{21}\pm 20(b-1)).$$
Also, we set $U(0)$ to be the empty sequence: so, for all $b\geq 0$ the sequence $U(b)$ contains $b$ elements
and $\supp(U(b))=[1,20b]$. 

Write $(\eta+1)\ell-N=20h_0+r_0$, where $0\leq r_0 < 20$, and define the sequence
$$S_0=(U(h_0), V_{r_0}\pm 20h_0 ).$$
Hence, $\supp(S_0\pm N)=[N+1, N+20h_0+21]\setminus \{ (\eta+1)\ell \}$.
For all $j\in [1,\lfloor t/2 \rfloor - \eta]$, write 
$\ell-21+r_{j-1}=20h_j +r_j$, 
where $0\leq  r_j < 20$, and define

$$S_j=\left(U(h_j)\pm\left(21j+20\sum_{i=0}^{j-1} h_i   \right),
V_{r_j}\pm  \left(21j+20\sum_{i=0}^j h_i  \right)\right).$$
Note that $(\eta+j+1)\ell-N=20\sum_{i=0}^{j} h_i+ 21j+ r_j $ and
$$\supp(S_j\pm N)=\left[N+1+21j+20\sum_{i=0}^{j-1} h_i,\;  N+21(j+1)+ 20\sum_{i=0}^j h_i  \right]
\setminus \{(\eta+j+1)\ell\}.$$
The elements of $\G$ are the first $\frac{m}{2}$ blocks in $\con\limits_{c=0}^{\lfloor t/2 \rfloor - \eta} (S_c\pm N)$.

Finally, writing $\F=(A_1,\ldots,A_{\frac{mq}{2}})$ and $\G=(G_1,\ldots,G_{\frac{m}{2}} )$,
for all $i=1,\ldots,\frac{m}{2}$, let $B_i$ be the block of size $2\times s$ obtained by juxtaposing the $q$ blocks 
$$A_{(i-1)q+1},\;A_{(i-1)q+2},\;A_{(i-1)q+3},\; \ldots,A_{iq}$$
and the  block $G_i$.
By construction, the sequence $\B=(B_1,\ldots,B_{\frac{m}{2}})$ satisfies condition \eqref{blocchi}, 
has cardinality $\frac{m}{2}$ and  
$\supp(\B)=\supp(\F)\cup \supp(\G)=[ 1,ms+\lfloor t/2 \rfloor ]
\setminus \{j\ell: j \in [1,\lfloor t/2\rfloor]\}$.
\end{proof}

\begin{figure}
\begin{footnotesize}
$$\begin{array}{rcl}
F_{3} & =& \begin{array}{|c|c|c|c|c|c|c|c|c|c|}\hline
1 & -4 & -10 & 16 & 14 & -17  &  19 & -20 & -25 & 26\\\hline
-2 & 5 & 7 & -13 & -8 &   11 &   -22 & 23 & 28 & -29\\\hline
\end{array}\;,\\[8pt]
F_{5} & =& \begin{array}{|c|c|c|c|c|c|c|c|c|c|}\hline
1 & -2 & 6 & -7 & -11 & 13  &  16  & -17 & -21 & 22\\\hline
-3 & 4 & -8 & 9 & 12 & -14  &  -18 & 19 & 23 & -24 \\\hline
\end{array}\;,\\[8pt]
 W_{7,3} & =& \begin{array}{|c|c|c|c|c|c|c|c|c|c|}\hline
2 & -23 & -15 & 7  & -8 & -12 & 19 & 20  & 16 & -6	\\ \hline
-1 & 22 & 13 & -5 & 9  & 11 & -21 & -18 & -14 & 4 \\ \hline  
\end{array}\;,\\[8pt]
W_{7,5}   & =& \begin{array}{|c|c|c|c|c|c|c|c|c|c|}\hline
2 & -18 & -11 & -13 & 7 & -4 & 21 & 22 & 10 & -16 \\ \hline
-1 & 17 & 9 & 15 & -6 & 3 & -23 & -20 & -8 & 14 \\ \hline
     \end{array}\;,\\[8pt]
W_{7,7}  & =& \begin{array}{|c|c|c|c|c|c|c|c|c|c|}\hline
2 & -23 & -19 & 20 & 16 & 12 & 9 & -3 & -4 & -10 \\ \hline
-1 & 22 & 17 & -18 & -15 & -13 & -11 & 5 & 6 & 8 \\ \hline
   \end{array}\;,\\[8pt] 
W_{9,5}   & =& \begin{array}{|c|c|c|c|c|c|c|c|c|c|}\hline
1 & -2 & 6 & -7 & -10 & 12 &  15 & -16 & -19 & 20 \\\hline 
-3 & 4 & -8 & 9 & 11 & -13 & -17 & 18 & 21 & -22 \\\hline 
   \end{array}\; ,\\[8pt]
   W_{9,9}    & =& \begin{array}{|c|c|c|c|c|c|c|c|c|c|}\hline
	1 & -19 & 20 & 13 & 6 & -8 & 15 & -2 & -14 & -12 \\ \hline
	-3 & 21 & -22 & -11 & -5 & 7 & -17 & 4 & 16 & 10 \\ \hline
      \end{array}\;,\\[8pt]
         W_{13,5}    & =& \begin{array}{|c|c|c|c|c|c|c|c|c|c|}\hline
	1 & -2 & 6 & -7 & -10 & 12 & 14 & -15 & -19 & 20 \\\hline 
 -3 & 4 & -8 & 9 & 11 & -13 & -16 & 17  &  21 & -22 \\\hline
          \end{array}\;,\\[8pt] 
          W_{15,3}   & =& \begin{array}{|c|c|c|c|c|c|c|c|c|c|}\hline
    4 &   -19 &   -13 &    16 &    -1 &    -9 &    20 &    12& 7 &   -17 \\ \hline
  -6 &    21 &    11 &  -14 &     2&    8 &       -22 &  -10 &    -5 &    15 \\\hline 
              \end{array}\; ,\\[8pt]
 W_{15 ,5}    & =& \begin{array}{|c|c|c|c|c|c|c|c|c|c|}\hline
1 & -17 & -14 & -8 & 22 & 6 & 16 & -2 & 11 & -15 \\ \hline
-3 & 19 & 12 & 10 & -21 & -7 & -18 & 4 & -9 & 13 \\ \hline
              \end{array}\;,\\[8pt] 
 V_1 & =& \begin{array}{|c|c|c|c|c|c|c|c|c|c|}\hline
2 & -3 & 6 & -7 & -10 & 12 & 14 & -15 & -18 & 19\\\hline 
-4 & 5 & -8 & 9 & 11 & -13 & -16& 17 & 20 & -21 \\\hline 
   \end{array}\;,\\[8pt]
  V_3 & =& \begin{array}{|c|c|c|c|c|c|c|c|c|c|}\hline
4 & 7 & -13 & 12 & -8 & -2 &  14 & -15 & -18 & 19 \\ \hline
-6 & -5 & 11 & -10 & 9 & 1 &  -16 & 17 & 20 & -21 \\ \hline
   \end{array}\;,\\[8pt]
   V_5 & =& \begin{array}{|c|c|c|c|c|c|c|c|c|c|}\hline
1 & -2 & 6 & -7 & -10 & 12 & 14 & -15 & -18 & 19 \\\hline 
-3 & 4 & -8 & 9 & 11 & -13 & -16 & 17& 20 & -21 \\\hline 
   \end{array}\;,\\[8pt]
V_7 & =& \begin{array}{|c|c|c|c|c|c|c|c|c|c|}\hline
1 & -2 & 8 & 11 & -5 & -13 & 14 & -15 & -18 & 19\\ \hline
-3 & 4 & -10 & -9 & 6 & 12 & -16 & 17 & 20 & -21 \\\hline
   \end{array}\;,\\[8pt]
 V_9 & =& \begin{array}{|c|c|c|c|c|c|c|c|c|c|}\hline
1 & -2 & 5 & -6 & -10 & 12 & 14  & -15 & -18 & 19\\\hline 
-3 & 4 & -7 & 8 & 11 & -13 & -16 &  17 & 20 & -21 \\\hline 
   \end{array}\;,\\[8pt]
V_{11}  & =& \begin{array}{|c|c|c|c|c|c|c|c|c|c|}\hline
1 & -19 & -7 & 16 & -9 & -13 & 18 &  17 & 4 & -8	\\\hline
-3 & 21 & 5 & -14 & 10 & 12 & -20 & -15 & -2 & 6 \\\hline
   \end{array}\;,\\[8pt]
V_{13}  & =& \begin{array}{|c|c|c|c|c|c|c|c|c|c|}\hline
1 & -2 & 5 & -6 & -9 & 11  & 14  & -15 & -18 & 19\\\hline 
-3 & 4 & -7 & 8 & 10 & -12 & -16 & 17 & 20 & -21 \\\hline 
   \end{array}\;,\\[8pt]
V_{15}  & =& \begin{array}{|c|c|c|c|c|c|c|c|c|c|}\hline
1 & -19 & -14 & 20 & 17 & 9 & 6 & -2 & -5 & -13 \\ \hline
-3 & 21 & 12 & -18 & -16 & -10 & -8 & 4 & 7 & 11 \\ \hline
   \end{array}\;,\\[8pt]
V_{17}  & =& \begin{array}{|c|c|c|c|c|c|c|c|c|c|}\hline
1 & -2 & 5 & -6 & -9 & 11 & 13  & -14 & -18 & 19\\\hline 
-3 & 4 & -7 & 8 & 10 & -12 & -15 & 16 & 20 & -21 \\\hline 
   \end{array}\;,\\[8pt]
V_{19}  & =& \begin{array}{|c|c|c|c|c|c|c|c|c|c|}\hline
1 & -16 & -17 & 14 & 21 & 5 & -9 & -2 & -8 & 11 \\ \hline
-3 & 18 & 15 & -12 & -20 & -6 & 7 & 4 & 10 & -13 \\ \hline
   \end{array}\; ,\\[8pt]
V_{21}   & =& \begin{array}{|c|c|c|c|c|c|c|c|c|c|}\hline
1 & -2 & 5 & -6 & -9 & 11 & 13  & -14 & -17 & 18\\\hline 
-3 & 4 & -7 & 8 & 10 & -12 & -15 & 16 & 19 & -20 \\\hline 
   \end{array}\;.
            \end{array}$$
\end{footnotesize}       
\caption{Shiftable blocks of size $2\times 10$ satisfying condition \eqref{blocchi}.}\label{bl10}
\end{figure}

\begin{cor}\label{seqB}
Let $m$ and $s$ be even integers with $m\geq 2$, $s\geq 6$ and $s \equiv 2 \pmod 4$.
Let $t$ be a divisor of $2ms$ and set $\ell=\frac{2ms}{t}+1$.
There exists a sequence $\B$ of cardinality $\frac{m}{2}$ such that $\B$ consists 
of blocks of size
$2\times s$, satisfies condition \eqref{blocchi} and 
$\supp(\B)=\left[ 1,ms+\left\lfloor t/2
\right\rfloor \right]
\setminus \left\{j\ell: j \in \left[1, \left\lfloor t/2 \right\rfloor \right]\right\}$.
\end{cor}

\begin{proof}
By the assumptions on $s$, we can write $s=2ph$, for a suitable odd prime $p$ and a positive integer $h$.
If $t$ divides $ms$, the result follows from Propositions \ref{G_rho-6}, \ref{G_rho-8} and \ref{G_rho-10}, 
depending on the residue class of $s$ modulo $6$.
It $t$ does not divide $ms$, then $t\equiv 0 \pmod 8$.
If $t$ divides $4mh=\frac{2ms}{p}$, then the  statement follows from Lemma \ref{non 8p}.
If $t$ does not divide $4mh$, then  $t$ is divisible by $p$ and so we can apply Lemma \ref{8p}.
\end{proof}

Now, we arrange the blocks so far constructed, proving the following.

\begin{prop}\label{s2}
Suppose  $4\leq s \leq n$, $4\leq k \leq m$ and $ms=nk$.
Let $t$ be a divisor of $2ms$. 
\begin{itemize}
\item[(1)] If $s\equiv 2\pmod 4$ and $k\equiv 0 \pmod 4$, then there exists an integer $\H_t(m,n;s,k)$ if and only if 
$m$ is even.
\item[(2)] If $s\equiv 0\pmod 4$ and $k\equiv 2 \pmod 4$, then there exists an integer $\H_t(m,n;s,k)$ if and only if 
$n$ is even.
 \end{itemize}
\end{prop}

\begin{proof}
(1) By Proposition \ref{prop:necc} we only have to prove the existence of an integer $\H_t(m,n;s,k)$ when 
$m$ is even. So, let $\B$ be the sequence obtained in Corollary \ref{seqB}.
Set $d=\gcd(\frac{m}{2}, n)$ and $a=\frac{sd}{n}$.
Note that $a$ is even integer. In fact, write $m =2 \bar m d $ and $n=d \bar n$. Since $k \equiv 0 \pmod 4$, 
from  $\frac{s}{2}\cdot \frac{m}{2}=n\frac{k}{4}$ we obtain that $\bar n$ divides $\frac{s}{2}$.

Given a block $B_h \in \B$, define for every $j\in [1,\bar n]$ the block
$T_j(B_h)$ of size $2\times a$ consisting of the columns $C_i$ of $B_h$ with $i\in [a(j-1)+1,a j]$.
So, the block $B_h$ of size $2\times s$ is obtained juxtaposing the blocks $T_1(B_h), T_2(B_h),\ldots,T_{\bar n}(B_h)$.
Furthermore, for all $i \in [1,\bar m]$ and all $j \in [1,\bar n]$, each of the sequences 
$$\left(T_j(B_{(i-1)d+1}),T_j(B_{(i-1)d+2}),\ldots,T_j(B_{id}) \right),$$
of cardinality $d$, satisfies condition \eqref{blocchi2}.

Let $A$ be an empty array of size $\bar m\times \bar n$. For every $i \in [1,\bar m]$ and  $j \in [1,\bar n]$,
replace the cell $(i,j)$ of $A$ 
with the block $P\left(T_j(B_{(i-1)d+1}),T_j(B_{(i-1)d+2}),\ldots,T_j(B_{id}) \right)$ according to the definition of 
page 
\pageref{bloccoP}.
Note that, for all $r\in [1,\frac{m}{2}]$, we have $\tau_r(A)=\tau_1(B_r)=0$ and  $\tau_{r+\frac{m}{2}}(A)=\tau_2(B_r)=0$.

By construction, $A$ is a p.f. array of size $m\times n$, its support coincides with $\supp(\B)$ and its rows and columns sum to zero.
Furthermore, each row contains  $a\bar n=s$ elements and each column contains $2a\bar m=k$ elements.
We conclude that $A$ is an integer $\H_t(m,n;s,k)$.\\
(2) It follows from (1). In fact, if $s\equiv 0 \pmod 4$ and $k\equiv 2\pmod 4$,
an integer $\H_t(m,n;s,k)$ can be obtained simply by taking the transpose of an integer $\H_t(n,m;k,s)$.
\end{proof}

We exhibit an integer $\H_{12}( 20, 15;  6, 8)$ and an integer $\H_{5}(6, 15;  10, 4)$  in Figure \ref{big}.
To help the reader, we highlighted the subarray $P(T_1(B_1),\ldots,T_1(B_5))$ in the first case, and the 
subarray $P(T_1(B_1),T_1(B_2),T_1(B_3))$ in the second case.

\begin{figure}[ht]
 \rotatebox{90}{
\begin{footnotesize}
$\begin{array}{|c|c|c|c|c |c|c|c|c|c| c|c|c|c|c|}\hline
   \omb   1&   \omb -2& \omb  &   \omb& \omb  &    5&    -6&    &   &   &   -9&    11&    &   &     \\\hline
   \omb   & \omb  13&  \omb -14&  \omb  & \omb  &   &   17&   -18&    &   &   &  -22&    24&    &     \\\hline
 \omb    & \omb  &  \omb 26&\omb   -27& \omb   &   &   &   30&   -31&    &   &   &  -34&    36&      \\\hline
   \omb   &  \omb&  \omb &  \omb 38&  \omb -39&    &   &   &   43&   -44&    &   &   &  -47&    49 \\\hline
   \omb  -52&\omb   &  \omb &\omb   & \omb  51&   -56&    &   &   &   55&    61&    &   &   &  -59 \\\hline
    \omb  -3& \omb    4&    \omb&  \omb & \omb &   -7&     8&    &   &   &   10&   -12&    &   &     \\\hline
  \omb    & \omb -15& \omb   16& \omb  & \omb  &   &  -19&    20&    &   &   &   23&   -25&    &     \\\hline
   \omb   &\omb   &  \omb-28&  \omb  29&   \omb &   &   &  -32&    33&    &   &   &   35&   -37&      \\\hline
   \omb   & \omb  & \omb  & \omb -40& \omb   41&    &   &   &  -45&    46&    &   &   &   48&   -50 \\\hline
    \omb  54& \omb   &  \omb &   \omb& \omb -53&    58&    &   &   &  -57&   -62&    &   &   &   60 \\\hline
      64&   -65&    &   &   &   68&   -69&    &   &   &  -72&    74&    &   &     \\\hline
      &   76&   -77&    &   &   &   80&   -81&    &   &   &  -85&    87&    &    \\\hline
      &   &   89&   -90&    &   &   &   93&   -94&    &   &   &  -97&    99&     \\\hline
      &   &   &  101&  -102&    &   &   &  106&  -107&    &   &   & -110&   112 \\\hline
    -115&    &   &   &  114&  -119&    &   &   &  118&   124&    &   &   & -122 \\\hline
     -66&    67&    &   &   &  -70&    71&    &   &   &   73&   -75&    &   &     \\\hline
      &  -78&    79&    &   &   &  -82&    83&    &   &   &   86&   -88&    &     \\\hline
      &   &  -91&    92&    &   &   &  -95&    96&    &   &   &   98&  -100&      \\\hline
      &   &   & -103&   104&    &   &   & -108&   109&    &   &   &  111&  -113 \\\hline
     117&    &   &   & -116&   121&    &   &   & -120&  -125&    &   &   &  123\\\hline
\end{array}$
\end{footnotesize} }
\qquad
\rotatebox{90}{
 \begin{footnotesize}
 $\begin{array}{|c|c|c|c|c| c|c|c|c|c| c|c|c|c|c|}\hline
   \omb 1&  \omb -2& \omb  &    5&   -6&   &   -9&   11&   &   13&  -14&   &  -17&   18&     \\ \hline
 \omb   & \omb  21&\omb  -22&   &   26&  -27&   &  -30&   32&   &   34&  -35&   &  -38&   39 \\\hline
\omb   -43&\omb   & \omb  42&  -47&   &   46&   53&   &  -51&  -56&   &   55&   60&   &  -59 \\\hline
   \omb -3& \omb   4&\omb   &   -7&    8&   &   10&  -12&   &  -15&   16&   &   19&  -20&    \\\hline
 \omb   &  \omb -23&\omb   24&   &  -28&   29&   &   31&  -33&   &  -36&   37&   &   40&  -41 \\\hline
    \omb 45&   \omb&\omb -44&   49&   &  -48&  -54&   &   52&   58&   &  -57&  -62&   &   61 \\\hline
 \end{array}$
\end{footnotesize}
 }
 \caption{An integer $\H_{12}( 20, 15;  6, 8)$ (on the left) and an integer $\H_{5}(6, 15;$ $ 10, 4)$ (on the right).} 
\label{big}
 \end{figure}

 \section{The case  $s,k\equiv 2 \pmod 4$}\label{s2k2}

The last case we consider is when $s,k\equiv 2\pmod 4$.
We also assume that $m$ is even: this means that  also $n$ must be even.
We start constructing sequences $\B$  satisfying the following property:
\begin{equation}\label{blocchiOLD}
\begin{array}{l}
\textrm{fixed } 2b \textrm{ integers } \sigma_1,\ldots,\sigma_{2b} \textrm{ such that }
\sum\limits_{i=1}^b \sigma_{2i-1} = \sum\limits_{i=1}^b \sigma_{2i} =0,\textrm{ the elements of } \B \\
\textrm{are shiftable blocks } B \textrm{ of size } 2\times 2b \textrm{ such that }
\tau_1(B)=\tau_2(B)=0 \textrm{ and } \gamma_{i}(B)=\sigma_i \\\textrm{for all } i \in [1,2b].
\end{array}
\end{equation}

\begin{lem}\label{G_rho-6-OLD}
Let $m$ and $s$ be even integers with $m\geq 2$, $s\geq 6$ and $s\equiv 2 \pmod 4$.
Let $t$ be a divisor of $2ms$ and set $\ell=\frac{2ms}{t}+1$.
There exists a sequence $\B'$ of $\frac{m}{2}$ shiftable blocks of size $2\times s$ such  that $\B'$ satisfies 
condition \eqref{blocchiOLD} and
$\supp(\B')=\left[ 1,ms+\left\lfloor t/2
\right\rfloor \right]
\setminus \left\{j\ell: j \in \left[1, \left\lfloor t/2 \right\rfloor \right]\right\}$.
\end{lem}

\begin{proof}
Suppose first that $t$ divides $ms$.
If $s\equiv 0\pmod 6$, consider the shiftable blocks of size $2\times 6$ given in Figure \ref{bl6-2}.
We can repeat the same proofs of Lemma \ref{F_rho} and Proposition \ref{G_rho-6}, using the blocks
of Figure \ref{bl6-2} instead of those of Figure \ref{bl6}.
In fact, each of these  new blocks satisfies \eqref{blocchiOLD} and has  the  following column sums:
$$F_i:\; (-1,3,-6,3,7,-6),\qquad V_j: \; (-1, 4, -1, -2, 2, -2).$$
Furthermore, the new blocks have the same support of the old ones.

The same can be done for $s\equiv 2 \pmod 6$, using the blocks of Figure \ref{bl8-2} and the proof of Proposition
\ref{G_rho-8}. Each of these new blocks satisfies \eqref{blocchiOLD} and  has the following column sums: 
$(-1,-8,-12,4,10,$ $ -1,3,5)$.
Finally, for $s\equiv 4\pmod 6$ we can repeat the proof of Proposition \ref{G_rho-10} using the blocks of Figure \ref{bl10-2}.
Each of these new blocks satisfies \eqref{blocchiOLD} and  has the following column
sums:
$$F_i=(-1,-1,1,7,6,3,-3,-2,-3,-7), \quad W_{j,i},V_j=(-1,4,-1,-2,2,-2,-2,2,2,-2).$$

Suppose now that $t$ does not divide $ms$.
In this case, repeat the proofs of Lemmas \ref{8p} and \ref{non 8p},
using the blocks of Figures \ref{blW3} and \ref{blW4} instead of those of Figures \ref{blW1}  and \ref{blW2}, 
respectively. Note that the blocks of Figure \ref{blW3} satisfy property \eqref{blocchiOLD} with column sums
 $(-2\ell, 2\ell, 2\ell, -2\ell)$ and $(-\ell, 4\ell, -\ell,$ $-2\ell, 2\ell, -2\ell)$.
The blocks $W_4,W_6$ of Figure \ref{blW4} satisfy property \eqref{blocchiOLD} with column sums $(-y,y,y,-y)$
 and $(-y, (p-2)y+1,2y,-(2p-1)y-2,-y,(p+1)y+1 )$,  respectively.

 In all these cases, we obtain a sequence $\B'$ that satisfies \eqref{blocchiOLD} and has the required support.
\end{proof}

\begin{figure}[hpt]
\begin{footnotesize}
  $$\begin{array}{rclrcl}
F_{3} & =& \begin{array}{|c|c|c|c|c|c|}\hline
1 & -4 & -11 & 16 & -10 & 8\\\hline
-2 & 7 & 5 & -13 & 17 & -14\\\hline
\end{array}\;, & 
F_5 & = &
\begin{array}{|c|c|c|c|c|c|}\hline
 1 &  -3 &  7 & -9 & -4 & 8\\\hline
-2 & 6 & -13 & 12& 11& -14 \\\hline
\end{array}\;.   \\[8pt]
V_1& =& \begin{array}{|c|c|c|c|c|c|c|}\hline
2 & -4 & -6 & -12 & 9 & 11 \\\hline
-3 & 8 & 5 & 10 & -7 & -13\\\hline
\end{array}\;, &
V_3& =& \begin{array}{|c|c|c|c|c|c|c|}\hline
1 & -9 & 11 & -7 & -4 & 8\\\hline
-2 & 13 & -12 & 5 & 6 &  -10\\\hline
\end{array}\;,\\[8pt]
V_5 & =& \begin{array}{|c|c|c|c|c|c|}\hline
-13 & 11 & 9 & -8 & -1 & 2\\\hline
12 & -7 & -10 & 6 & 3 & -4\\\hline
    \end{array}\;,  &
V_7 & =& \begin{array}{|c|c|c|c|c|c|}\hline
1 & 12 & 9 & -5 & -4 & -13 \\\hline
-2 & -8 & -10 & 3 & 6 & 11\\\hline
\end{array}\;,\\[8pt]
V_9 & =& \begin{array}{|c|c|c|c|c|c|}\hline
1 & -3 & -5 & 6 & 13 & -12\\\hline
-2 & 7 & 4 & -8 & -11 & 10 \\\hline
    \end{array}\;, &
 V_{11} & =& \begin{array}{|c|c|c|c|c|c|}\hline
 1 & -9 & 7 & -6 & -3 & 10\\\hline
-2 & 13 & -8 & 4 & 5 & -12\\\hline
    \end{array}\;,\\[8pt]
V_{13}& =& \begin{array}{|c|c|c|c|c|c|c|}\hline
1 & -3 & -5 & 6  & 12 & -11\\\hline
-2 & 7 & 4 &  -8 & -10 & 9\\\hline
\end{array}\;.
\end{array}$$
\end{footnotesize}
\caption{Shiftable blocks of size $2\times 6$ satisfying condition \eqref{blocchiOLD}.}\label{bl6-2}
\end{figure}

\begin{figure}[hpt]
\begin{footnotesize}
 $$\begin{array}{rcl}
F_{3} & =& \begin{array}{|c|c|c|c|c|c|c|c|}\hline
     1 &   11 &  -17 &  -16 &   14 &   22 &   -7 &   -8 \\\hline
    -2 &  -19 &    5 &   20 &   -4 &  -23 &   10 &   13 \\\hline
\end{array}\;,\\[8pt]
F_5 & = & \begin{array}{|c|c|c|c|c|c|c|c|}\hline
       1 &    3 &  -19 &   16 &   -4 &   17 &   -6 &   -8 \\\hline
      -2 &  -11 &    7 &  -12 &   14 &  -18 &    9 &   13 \\\hline
      \end{array}\;,\\[8pt]
      W_{7,3} & =& \begin{array}{|c|c|c|c|c|c|c|c|}\hline
       1 &  8 &-18 & -9 & 15 &-12 & -4 & 19 \\\hline
      -2 &-16 &  6 & 13 & -5 & 11 &  7 &-14 \\\hline
         \end{array}\;,\\[8pt]
W_{7,5} & = & \begin{array}{|c|c|c|c|c|c|c|c|}\hline
       1 &  9 &-16 & 18 & 13 & -7 & -8 &-10 \\\hline
      -2 &-17 &  4 &-14 & -3 &  6 & 11 & 15 \\\hline
         \end{array}\;,\\[8pt]
W_{7,7}  & = & \begin{array}{|c|c|c|c|c|c|c|c|}\hline
       1 &  3 &-17 & -4 & -6 &-10 & 15 & 18 \\\hline
      -2 &-11 &  5 &  8 & 16 &  9 &-12 &-13\\\hline
         \end{array}\;,\\[8pt]
W_{9,5} & =& \begin{array}{|c|c|c|c|c|c|c|c|}\hline
     1 &-17 &-15 & 12 & 16 &-11 & -4 & 18\\\hline
    -2 &  9 &  3 & -8 & -6 & 10 &  7 &-13\\\hline
               \end{array}\;,\\[8pt]
W_{11,3}  & =&  \begin{array}{|c|c|c|c|c|c|c|c|}\hline
  1 &-12 &-17 & 15 & -6 & -9 & 10 & 18 \\\hline
  -2 &  4 &  5 &-11 & 16 &  8 & -7 &-13 \\\hline
                \end{array}\;,\\[8pt]
W_{11,5} & =&  \begin{array}{|c|c|c|c|c|c|c|c|}\hline
   1 &  6 &-15 & -4 & -7 & 10 & -9 & 18\\\hline
  -2 &-14 &  3 &  8 & 17 &-11 & 12 &-13 \\\hline
                \end{array}\;,\\[8pt]
V_1 & = & \begin{array}{|c|c|c|c|c|c|c|c|}\hline
       2 &  8 &  5 & -7 & -4 & 12 & -6 &-10  \\\hline
      -3 &-16 &-17 & 11 & 14 &-13 &  9 & 15 \\\hline
      \end{array}\;,\\[8pt]
  V_3 & = & \begin{array}{|c|c|c|c|c|c|c|c|}\hline
       1 &  6 &-16 & -7 & -5 & -9 & 13 & 17  \\\hline
      -2 &-14 &  4 & 11 & 15 &  8 &-10 &-12 \\\hline
         \end{array}\;,\\[8pt]
         V_5 & = & \begin{array}{|c|c|c|c|c|c|c|c|}\hline
       1 &-11 &  4 & -8 & 17 & 13 & -6 &-10 \\\hline
      -2 &  3 &-16 & 12 & -7 &-14 &  9 & 15  \\\hline
         \end{array}\;,\\[8pt]
V_7 & = & \begin{array}{|c|c|c|c|c|c|c|c|}\hline
       1 &-17 &-16 & -6 & 15 & 12 & 14 & -3  \\\hline
      -2 &  9 &  4 & 10 & -5 &-13 &-11 &  8  \\\hline
         \end{array}\;,\\[8pt]
V_9 & = & \begin{array}{|c|c|c|c|c|c|c|c|}\hline
       1 &-11 &  5 & 12 & -6 & 13 & -4 &-10  \\\hline
      -2 &  3 &-17 & -8 & 16 &-14 &  7 & 15  \\\hline
         \end{array}\;,\\[8pt]
V_{11} & = & \begin{array}{|c|c|c|c|c|c|c|c|}\hline
       1 &  5 &-16 & -6 & -7 & 14 & 12 & -3  \\\hline
      -2 &-13 &  4 & 10 & 17 &-15 & -9 &  8  \\\hline
         \end{array}\;,\\[8pt]
V_{13} & = & \begin{array}{|c|c|c|c|c|c|c|c|}\hline
       1 &  9 &-15 & 16 & -4 &  6 & -8 & -5  \\\hline
      -2 &-17 &  3 &-12 & 14 & -7 & 11 & 10  \\\hline
         \end{array}\;,\\[8pt]
V_{15} & = & \begin{array}{|c|c|c|c|c|c|c|c|}\hline
       1 &  9 &  4 & 12 & -3 & -7 &-11 & -5  \\\hline
      -2 &-17 &-16 & -8 & 13 &  6 & 14 & 10 \\\hline
         \end{array}\;,\\[8pt]
V_{17} & = & \begin{array}{|c|c|c|c|c|c|c|c|}\hline
       1 &  7 &  4 & -6 & -3 & 11 & -5 & -9  \\\hline
      -2 &-15 &-16 & 10 & 13 &-12 &  8 & 14  \\\hline
        \end{array}\;.
\end{array}$$
\end{footnotesize}
\caption{Shiftable blocks of size $2\times 8$ satisfying condition \eqref{blocchiOLD}.}\label{bl8-2}
\end{figure}

\begin{figure}[hpt]
\begin{footnotesize}
$$\begin{array}{rcl}
F_{3} & =& \begin{array}{|c|c|c|c|c|c|c|c|c|c|}\hline
   1 &  4 & -7 &-13 &-10 & 28 & 23 & 17 &-14 &-29 \\\hline
  -2 & -5 &  8 & 20 & 16 &-25 &-26 &-19 & 11 & 22 \\\hline
\end{array}\;,\\[8pt]
F_{5} & =& \begin{array}{|c|c|c|c|c|c|c|c|c|c|}\hline
   1 &  -8 & 4 & -6 &23 & 24 & -12 & 14 &-22 &-18  \\\hline
  -2 & 7 &  -3 & 13 & -17 &-21 & 9 &-16 &  19 & 11  \\\hline
\end{array}\;,\\[8pt]
 W_{7,3} & =& \begin{array}{|c|c|c|c|c|c|c|c|c|c|}\hline
      1 & -15 &  -5 &  21 &  18 &  12 & -13 &  -6 &   9 & -22 \\\hline
     -2 &  19 &   4 & -23 & -16 & -14 &  11 &   8 &  -7 &  20 \\\hline
     \end{array}\;,\\[8pt]
 W_{7,5}   & =& \begin{array}{|c|c|c|c|c|c|c|c|c|c|}\hline
      1 & -10 &   3 &  21 &  22 &  11 &  -8 &  -7 & -15 & -18\\\hline
     -2 &  14 &  -4 & -23 & -20 & -13 &   6 &   9 &  17 &  16 \\\hline
     \end{array}\;,\\[8pt]
 W_{7,7}  & =& \begin{array}{|c|c|c|c|c|c|c|c|c|c|}\hline
      1 & -19 & -13 &  20 &  18 &   9 &   8 &  -3 &  -4 & -17 \\\hline
     -2 &  23 &  12 & -22 & -16 & -11 & -10 &   5 &   6 &  15 \\\hline
   \end{array}\;,\\[8pt]
W_{9,5}   & =& \begin{array}{|c|c|c|c|c|c|c|c|c|c|}\hline
      1 &   -3 &  -13 &   20 &   -4 &   16 &    9 &   10 &  -15 &  -21  \\\hline
     -2 &    7 &   12 &  -22 &    6 &  -18 &  -11 &   -8 &   17 &   19  \\\hline
   \end{array}\; ,\\[8pt]
   W_{9,9}    & =& \begin{array}{|c|c|c|c|c|c|c|c|c|c|}\hline
      1 &   -7 &  -17 &  -14 &  -13 &   20 &    8 &    6 &   -3 &   19  \\\hline
     -2 &   11 &   16 &   12 &   15 &  -22 &  -10 &   -4 &    5 &  -21  \\\hline
      \end{array}\;,\\[8pt]
         W_{13,5}    & =& \begin{array}{|c|c|c|c|c|c|c|c|c|c|}\hline
    1 &   -3 &  -17 &  -15 &    6 &  -14 &   20 &   -8 &   11 &   19 \\\hline
       -2 &    7 &   16 &   13 &   -4 &   12 &  -22 &   10 &   -9 &  -21 \\\hline
          \end{array}\;,\\[8pt]
          W_{15,3}   & =& \begin{array}{|c|c|c|c|c|c|c|c|c|c|}\hline
       1 &  -12 &   13 &   20 &   17 &   -6 &    5 &   -9 &   -8 &  -21 \\\hline
       -2 &   16 &  -14 &  -22 &  -15 &    4 &   -7 &   11 &   10 &   19\\\hline
              \end{array}\; ,\\[8pt]
 W_{15 ,5}    & =& \begin{array}{|c|c|c|c|c|c|c|c|c|c|}\hline
        1 &   22 &    3 &   19 &   13 &   -8 &   -9 &  -10 &  -15 &  -16 \\\hline
       -2 &  -18 &   -4 &  -21 &  -11 &    6 &    7 &   12 &   17 &   14 \\\hline
              \end{array}\;,\\[8pt]
 V_1 & =& \begin{array}{|c|c|c|c|c|c|c|c|c|c|}\hline
  2 &-4 &-6 &-12 &9 &11 &  14 &-15 & -18 &19 \\\hline
  -3 &8 &5 &10 &-7 &-13 &  -16 &17 & 20 & -21 \\\hline
   \end{array}\;,\\[8pt]
  V_3 & =& \begin{array}{|c|c|c|c|c|c|c|c|c|c|}\hline
  1 &-9 &11 &-7 &-4 &8 &  14 &-15 & -18 &19 \\\hline
  -2 &13 &-12 &5 &6 &-10 &  -16 &17 & 20 & -21 \\\hline
   \end{array}\;,\\[8pt]
   V_5 & =& \begin{array}{|c|c|c|c|c|c|c|c|c|c|}\hline
   -13 &11 &9 &-8 &-1 &2 &  14 &-15 & -18 &19 \\\hline
   12 &-7 &-10 &6 &3 &-4 &  -16 &17 & 20 & -21 \\\hline
   \end{array}\;,\\[8pt]
V_7 & =& \begin{array}{|c|c|c|c|c|c|c|c|c|c|}\hline
   1 &12 &9 &-5 &-4 &-13 &  14 &-15 & -18 &19 \\\hline
   -2 &-8 &-10 &3 &6 &11 &  -16 &17 & 20 & -21 \\\hline
   \end{array}\;,\\[8pt]
 V_9 & =& \begin{array}{|c|c|c|c|c|c|c|c|c|c|}\hline
   1 &-3 &-5 &6 &13 &-12 &  14 &-15 & -18 &19 \\\hline
   -2 &7 &4 &-8 &-11 &10 &  -16 &17 & 20 & -21 \\\hline
   \end{array}\;,\\[8pt]
V_{11}  & =& \begin{array}{|c|c|c|c|c|c|c|c|c|c|}\hline
   1 &-9 &7 &-6 &-3 &10 &  14 &-15 & -18 &19 \\\hline
   -2 &13 &-8 &4 &5 &-12 &  -16 &17 & 20 & -21 \\\hline
   \end{array}\;,\\[8pt]
V_{13}  & =& \begin{array}{|c|c|c|c|c|c|c|c|c|c|}\hline
   1 &-3 &-5 &6 &12 &-11 &  14 &-15 & -18 &19 \\\hline
   -2 &7 &4 &-8 &-10 &9 &  -16 &17 & 20 & -21 \\\hline
   \end{array}\;,\\[8pt]
V_{15}  & =& \begin{array}{|c|c|c|c|c|c|c|c|c|c|}\hline
       1 &  -17 &   19 &   16 &   13 &    3 &   -6 &   -8 &   -7 &  -14 \\\hline
      -2 &   21 &  -20 &  -18 &  -11 &   -5 &    4 &   10 &    9 &   12  \\\hline
   \end{array}\;,\\[8pt]
V_{17}  & =& \begin{array}{|c|c|c|c|c|c|c|c|c|c|}\hline
   1 &-3 &-5 &6 &12 &-11 &  13 &-14 & -18 &19 \\\hline
   -2 &7 &4 &-8 &-10 &9 &  -15 &16 & 20 & -21 \\\hline
   \end{array}\;,\\[8pt]
V_{19}  & =& \begin{array}{|c|c|c|c|c|c|c|c|c|c|}\hline
       1 &  -17 &    3 &   18 &   16 &   13 &   -7 &   -6 &  -10 &  -11  \\\hline
      -2 &   21 &   -4 &  -20 &  -14 &  -15 &    5 &    8 &   12 &    9  \\\hline
   \end{array}\; ,\\[8pt]
V_{21}   & =& \begin{array}{|c|c|c|c|c|c|c|c|c|c|}\hline
   1 &-3 &-5 &6 &12 &-11 &  13 &-14 & -17 &18 \\\hline
   -2 &7 &4 &-8 &-10 &9 &  -15 &16 & 19 & -20 \\\hline
   \end{array}\;.
            \end{array}$$
\end{footnotesize}
\caption{Shiftable blocks of size $2\times 10$ satisfying condition \eqref{blocchiOLD}.}\label{bl10-2}
\end{figure}

\begin{figure}[hpt]
\begin{footnotesize} 
 $$\begin{array}{rcl}
W_4 & =& \begin{array}{|c|c|c|c|}\hline
       1 &   -(\ell+1) &  -(4\ell+1) &    5\ell+1\\\hline
   -(2\ell+1) &  3\ell+1 &    6\ell+1 & -(7\ell+1)\\\hline
\end{array}\;,\\[8pt]
W_6 & = &
\begin{array}{|c|c|c|c|c|c|}\hline
   1 &     -(2\ell+1) &   -(4\ell+1) &      5\ell+1 &     11\ell+1 &  -(10\ell+1) \\\hline
 -(\ell+1) &      6\ell+1 &      3\ell+1 &   -(7\ell+1) &  -(9\ell+1) &      8\ell+1 \\\hline
\end{array}\;.
\end{array}$$
\end{footnotesize}
\caption{Shiftable blocks satisfying condition \eqref{blocchiOLD}.}\label{blW3}
\end{figure}

\begin{figure}[hpt]
\begin{footnotesize}
   $$\begin{array}{rcl}
W_4 & =& \begin{array}{|c|c|c|c|}\hline
   y+1 &    -((p+1)y+2) &  -((2p+1)y+3) &  (3p+1)y+4\\\hline
 -(2y+1) &    (p+2)y+2 &    (2p+2)y+3 &  -((3p+2)y+4) \\\hline
\end{array}\;,\\[8pt]
W' & = &
\begin{array}{|c|c|c|c|}\hline
      1 &  -(2y+1) &  (3p+2)y+4 & -((3p+1)y+4) \\\hline
 -(y+1) &  py+2 &     -(3py+4)  &    (p+2)y+2 \\\hline
\end{array}\;, \\[8pt]
W'' & = &
\begin{array}{|c|c|}\hline
 -((2p+1)y+3) &  (2p+2)y+3 \\\hline
      2py+3 &  -((p+1)y+2) \\\hline
\end{array}\;, \\[8pt]
W_6 & = &
\begin{array}{|c|c|}\hline
 W' & W''\\\hline
\end{array}\;.
\end{array}$$

\end{footnotesize}
 \caption{Shiftable blocks $W_4,W_6$ satisfying condition \eqref{blocchiOLD}.}\label{blW4}
\end{figure}

\begin{prop}\label{sk2}
Suppose $6\leq s \leq n$, $6\leq k \leq m$, $ms=nk$ and $s,k\equiv 2 \pmod 4$.
If $m$ is even, there exists a shiftable integer $\H_t(m,n;s,k)$ for every
divisor $t$ of $2ms$.
\end{prop}

\begin{proof}
Without loss of generality, we may assume $m\geq n$ (and so $s\leq k$).
Let $\B=(B_1,\ldots,B_{\frac{m}{2}})$ and $\B'=(B'_1,\ldots,B'_{\frac{m}{2}})$  be the sequences of blocks of size 
$2\times s$ constructed in
Corollary \ref{seqB} and Lemma \ref{G_rho-6-OLD}, respectively, where $\B$ satisfies \eqref{blocchi}, $\B'$ satisfies 
\eqref{blocchiOLD} and
$\supp(\B)=\supp(\B')=[ 1,ms+\left\lfloor t/2 \right\rfloor]
\setminus \{j\ell : j \in [1, \lfloor t/2 \rfloor ]$ with $\ell=\frac{2ms}{t}+1$.
Set
$$\wt \B= \left(B_{\frac{n}{2}+1},\ldots, B_{\frac{m}{2}} \right) \equad \wt \B'=\left(B_1',\ldots,B'_{\frac{n}{2}} 
\right).$$
Since $\supp(B_i)=\supp(B'_i)$ for all $i\in [1,\frac{m}{2}]$, it follows that
$$\supp(\wt \B')=
\left[1, sn  +\left\lfloor \frac{sn}{\ell-1} \right\rfloor \right]\setminus
\left\{ j\ell : j \in \left[1, \left\lfloor\frac{sn}{\ell-1} \right\rfloor \right]\right\}$$
and
$$\begin{array}{rcl}
\supp(\wt \B) & =&   \supp(\B) \setminus \supp(\wt \B')\\
& = &\left[sn  +\left\lfloor \frac{sn}{\ell-1} \right\rfloor +1,
ms+\left\lfloor \frac{t}{2}\right\rfloor\right] \setminus 
\left\{j\ell: j \in \left[  \left\lfloor\frac{sn}{\ell-1} \right\rfloor + 1,
\left\lfloor \frac{t}{2}\right\rfloor\right]\right\}.
\end{array}$$
Hence,  $\supp(\wt \B'\con \wt \B)=[ 1,ms+\left\lfloor t/2 \right\rfloor]
\setminus \{j\ell : j \in [1, \lfloor t/2 \rfloor ]$.

Using the blocks of the sequence $\wt \B'$, we first construct a square shiftable p.f. array $A_1$ of size $n$
such that each row and each column contains $s$ filled cells and such that the elements in every row and column sum to 
zero. Hence, take an empty array  $A_1$ of size $n\times n$ and arrange the $\frac{n}{2}$ blocks $B'_r=(b_{i,j}^{(r)})$ 
of $\wt \B'$ in 
such a way that the element $b_{1,1}^{(r)}$ fills the cell $(2r-1,2r-1)$ of $A_1$.
This process makes $A_1$ a p.f. array with $s$ filled cells in each row and in each column.
Since the rows of the blocks $B_r'$ sum to zero, also the rows of $A_1$ sum to zero.
Looking at the columns, the $s$ elements of a column of $A_1$ are
$$b_{1,s}^{(r)}, b_{2,s}^{(r)},\; b_{1,s-2}^{(r+1)}, b_{2,s-2}^{(r+1)},\; 
b_{1,s-4}^{(r+2)}, b_{2,s-4}^{(r+2)}, \;\ldots,\;  b_{1,2}^{(r+s/2)}, b_{2,2}^{(r+s/2)}$$
or
$$b_{1,s-1}^{(r)}, b_{2,s-1}^{(r)},\; b_{1,s-3}^{(r+1)}, b_{2,s-3}^{(r+1)},\; 
b_{1,s-5}^{(r+2)}, b_{2,s-5}^{(r+2)}, \;\ldots,\;  b_{1,1}^{(r+s/2)}, b_{2,1}^{(r+s/2)},$$
where the exponents $r,\ldots,r+s/2$ must be read modulo $\frac{n}{2}$.
Since $\wt \B'$ satisfies condition \eqref{blocchiOLD}, the sum of these elements is
$$\sum_{j=1}^{s/2}\sigma_{2j}=0 \quad \textrm{ or }\quad  \sum_{j=1}^{s/2}\sigma_{2j-1}=0, \quad \textrm{ respectively.}$$
By construction, $\supp(A_1)=\supp(\wt \B')$.

Now, if $m=n$, then $A_1$ is actually a shiftable integer $\H_t(m,n;k,s)$.
Suppose that $m>n$. If we arrange the blocks of the sequence $\wt \B$
mimicking what we did for the construction of an integer $\H_1(m-n, n; s, k-s)$  in
the proof of Proposition \ref{s2},
we obtain a shiftable p.f. array $A_2$ of size $(m-n)\times n$ such that
$\supp(A_2)=\wt \B$, rows and columns sum to zero, each row contains $s$ filled cells and each column  contains $k-s$ filled cells.
Let $A$ be the p.f. array of size $m\times n$ obtained taking 
$$A=\begin{array}{|c|}\hline
   A_1\\\hline
   A_2 \\\hline
   \end{array}.$$
Each row of  $A$ contains $s$ filled cells and each of its columns contains $s+(k-s)=k$ filled cells.
Since  $\supp(\wt \B'\con \wt \B)=[ 1,ms+\left\lfloor t/2 \right\rfloor]
\setminus \{j\ell : j \in [1, \lfloor t/2 \rfloor ]$,  $A$ is a shiftable integer $\H_t(m,n;s,k)$.
\end{proof}

An integer  $\H_{32}(16;14)$   and an integer $\H_{15}(20,12; 6, 10)$ are shown  in Figure \ref{big2}.

\begin{figure}
\rotatebox{90}{
\begin{footnotesize}
$\begin{array}{|c|c|c|c|c|c|c|c|c|c|c|c|c|c|c|c|}\hline
      1&    -3&    -5&     6&    12&   -11&   103&   108&  -118&  -109&  -107&  -111&   115&   119&    &     \\ \hline
      -2&     7&     4&    -8&   -10&     9&  -104&  -116&   106&   113&   117&   110&  -112&  -114&    &     \\ \hline
      &   &   13&   -21&    23&   -19&   -16&    20&   121&   129&   124&   132&  -123&  -127&  -131&  -125 \\ \hline
      &   &  -14&    25&   -24&    17&    18&   -22&  -122&  -137&  -136&  -128&   133&   126&   134&   130 \\ \hline
    -145&  -142&    &   &  -38&    36&    34&   -33&   -26&    27&   138&   146&  -152&   153&  -141&   143 \\ \hline
     148&   147&    &   &   37&   -32&   -35&    31&    28&   -29&  -139&  -154&   140&  -149&   151&  -144 \\ \hline
    -161&   168&   166&  -157&    &   &   39&    50&    47&   -43&   -42&   -51&   155&   159&  -170&  -160 \\ \hline
     171&  -169&  -163&   162&    &   &  -40&   -46&   -48&    41&    44&    49&  -156&  -167&   158&   164 \\ \hline
     176&   183&  -177&   184&  -175&  -181&    &   &   52&   -54&   -56&    57&    64&   -63&   172&  -182 \\ \hline
    -188&  -179&   187&  -185&   178&   186&    &   &  -53&    58&    55&   -59&   -62&    61&  -173&   174 \\ \hline
     189&  -205&  -204&  -194&   203&   200&   202&  -191&    &   &   65&   -73&    71&   -70&   -67&    74 \\ \hline
    -190&   197&   192&   198&  -193&  -201&  -199&   196&    &   &  -66&    77&   -72&    68&    69&   -76 \\ \hline
      89&   -88&   206&  -216&   209&  -213&   222&   218&  -211&  -215&    &   &   78&   -80&   -82&    83 \\ \hline
     -87&    86&  -207&   208&  -221&   217&  -212&  -219&   214&   220&    &   &  -79&    84&    81&   -85 \\ \hline
     -95&    96&   102&  -101&   223&   228&  -238&  -229&  -227&  -231&   235&   239&    &   &   91&   -93 \\ \hline
      94&   -98&  -100&    99&  -224&  -236&   226&   233&   237&   230&  -232&  -234&    &   &  -92&    97 \\ \hline
\end{array}$
\end{footnotesize}}
\quad 
\rotatebox{90}{
\begin{footnotesize}
$\begin{array}{|c|c|c|c|c| c|c|c|c|c| c|c|}\hline
      1&    -3&    -5&     6&    12&   -11&    &   &   &   &   &     \\ \hline
      -2&     7&     4&    -8&   -10&     9&    &   &   &   &   &     \\ \hline
      &   &  -25&    23&    21&   -20&   -13&    14&    &   &   &     \\ \hline
      &   &   24&   -19&   -22&    18&    15&   -16&    &   &   &     \\ \hline
      &   &   &   &   26&   -28&   -30&    31&    38&   -37&    &     \\ \hline
      &   &   &   &  -27&    32&    29&   -33&   -36&    35&    &     \\ \hline
      &   &   &   &   &   &   39&   -41&   -43&    44&    50&   -49 \\ \hline
      &   &   &   &   &   &  -40&    45&    42&   -46&   -48&    47 \\ \hline
      63&   -62&    &   &   &   &   &   &   52&   -54&   -56&    57 \\ \hline
     -61&    60&    &   &   &   &   &   &  -53&    58&    55&   -59 \\ \hline
      72&   -71&   -64&    65&    &   &   &   &   &   &  -76&    74 \\ \hline
     -73&    69&    66&   -67&    &   &   &   &   &   &   75&   -70 \\ \hline
      77&   -78&    &   &   81&   -82&    &   &  -86&    88&    &     \\ \hline
      &   90&   -91&    &   &   94&   -95&    &   &  -98&   100&      \\ \hline
      &   &  103&  -104&    &   &  107&  -108&    &   & -111&   113 \\ \hline
    -116&    &   &  115&  -121&    &   &  120&   126&    &   & -124 \\ \hline
     -79&    80&    &   &  -83&    84&    &   &   87&   -89&    &     \\ \hline
      &  -92&    93&    &   &  -96&    97&    &   &   99&  -101&      \\ \hline
      &   & -105&   106&    &   & -109&   110&    &   &  112&  -114 \\ \hline
     118&    &   & -117&   123&    &   & -122&  -127&    &   &  125 \\ \hline
\end{array}$
\end{footnotesize}} 
\caption{An integer  $\H_{32}(16;14)$  (on the left) and an  integer $\H_{15}(20,12;$ $6, 10)$ (on the right).} 
\label{big2}
\end{figure}

\section{Conclusions}

As the reader could observe, our constructions are all obtained taking basic blocks and arranging
these blocks in some order. In all cases, changing the order of the blocks but keeping the same skeleton,
one can obtain a different $\H_t(m,n;s,k)$ for a fixed choice of $(m,n,s,k,t)$.
Hence, for each case we actually produced at least $\left\lfloor \frac{m}{2}\right\rfloor!$
different integer $\H_t(m,n;s,k)$ with the same skeleton.
Our procedures can be easily implemented in a computer: programs for GAP are available upon request writing to
the second author.
We also point out that the Heffter arrays $\H_t(n;k)$ given here for $t=1,2,k$ are
actually different from the arrays
obtained in \cite{ADDY,RelH}.

As remarked in the introduction, (simple) relative Heffter arrays $\H_t(m,n;s,k)$ can be used for exhibiting pairs of
orthogonal cyclic  decompositions of the complete multipartite graph $K_{\frac{2ms+t}{t}\times t}$, where
one decomposition consists of $s$-cycles and the other one consists of $k$-cycles.
The reader  interested in this type of problems can find full details in \cite{RelH}.

\end{document}